\documentclass[12pt]{amsart}
\usepackage{amssymb}
\usepackage{amsfonts}
\usepackage{graphicx}
\usepackage{amsmath}
\usepackage{xcolor}

\newtheorem{theorem}{Theorem}

\newtheorem{lemma}[theorem]{Lemma}

\newtheorem{proposition}[theorem]{Proposition}

\begin{document}
\title[]{Fischer decompositions for entire functions and the Dirichlet
problem for parabolas}
\author{H. Render and J. M. Aldaz}

\address{H. Render: School of Mathematical Sciences, University College
	Dublin, Dublin 4, Ireland.}
\email{hermann.render@ucd.ie}
\address{J.M. Aldaz:
 Instituto de Ciencias Matem\'aticas (CSIC-UAM-UC3M-UCM) and Departamento de  Matem\'aticas,
	Universidad  Aut\'onoma de Madrid, Cantoblanco 28049, Madrid, Spain.}
\email{jesus.munarriz@uam.es}
\email{jesus.munarriz@icmat.es}


%\thanks{The first named author was partially supported by Grant MTM2015-65792-P of the
	%MINECO of Spain, and also by by ICMAT Severo Ochoa project SEV-2015-0554 (MINECO)}

\maketitle

\begin{abstract}
Let $P_{2k}$ be a homogeneous polynomial of degree $2k$ and assume that
there exist $C>0$, $D>0$ and $\alpha \ge 0$ such that 
\begin{equation*}
\left\langle P_{2k}f_{m},f_{m}\right\rangle _{L^2(\mathbb{S}^{d-1})}\geq 
\frac{1}{C\left( m+D\right) ^{\alpha }}\left\langle f_{m},f_{m}\right\rangle
_{\mathbb{S}^{d-1}}
\end{equation*}%
for all homogeneous polynomials $f_{m}$ of degree $m.$ Assume that $P_{j}$
for $j=0, \dots ,\beta <2k$ are homogeneous polynomials of degree $j$. The
main result of the paper states that for any entire function $f$ of order $%
\rho <\left( 2k-\beta \right) /\alpha $ there exist entire functions $q$ and 
$h$ of order bounded by $\rho$ such that 
\begin{equation*}
f=\left( P_{2k}-P_{\beta }- \dots -P_{0}\right) q+h\text{ and }\Delta
^{h}r=0.
\end{equation*}%
This result is used to establish the existence of entire harmonic solutions
of the Dirichlet problem for parabola-shaped domains on the plane, with data
given by entire functions of order smaller than $\frac{1}{2}$.
\end{abstract}

\section{Introduction}

This paper is dedicated to the memory of Harold S. Shapiro (1928 - 2021),
whose scientific work has had a very high and long-lasting impact on the
mathematical community, easily recognizable in recent research. For a
description of H. S. Shapiro's main achievements, the interested reader is
referred to \cite{Khav22}.

Here we discuss a rather specific topic, studied and promoted by H. S.
Shapiro in his celebrated paper \cite{Shap89} (see also \cite{NeSh66} and 
\cite{NeSh68}), written in collaboration with \emph{D.J. Newman}). In \cite%
{Shap89} Shapiro introduced the notion of the \emph{Fischer decomposition}
of an entire function, which goes back to an old paper of Ernst Fischer 
\footnote{%
E. Fischer made important contributions to the development of abstract
Hilbert spaces (e.g. the Riesz-Fischer theorem), so, not surprisingly,
Hilbert space arguments play an important role in his paper. See \cite{Gonz}%
, p. 217 and p. 228, for some biographical comments and remarks.} from 1917
discussing the polynomial case, see \cite{Fisc17}. This method has many
applications: in \cite{EbSh95} and \cite{EbSh96} it was used to generalize
the \emph{Cauchy--Kowaleskaya theorem }to a much wider setting (see \cite%
{EbRe08} and \cite{EbRe08b} for further developments). In \cite{KhSh92} this
approach was used to explore analytical extension properties of solutions of
the Dirichlet problem for entire data, with respect to an ellipsoid and to
more general domains, see also \cite{Armi04}, \cite{AGV03}, \cite{AxRa95}, 
\cite{Ba}, \cite{ChSi01}, \cite{Eben92} \cite{EKS05} and \cite{HaSh94}. In 
\cite{KhSh92} the \emph{Khavinson-Shapiro conjecture} was stated, which
initiated a wide range of research activity, see (in alphabetical order) 
\cite{KL}, \cite{KhSt10}, \cite{Lund09}, \cite{LuRe11}, \cite{PutSty}, \cite%
{Rend08}, \cite{Rend16} and \cite{Rend17}. Further work on Fischer
decompositions can be found in \cite{ElNa12}, \cite{Khav97}, \cite{MeSt85}
and \cite{MeYg92}. In passing we note that Fischer decompositions also play
an important role in Clifford Analysis, see e.g. the \cite{BDS10} and the
references therein.

Let us recall some notations and terminology in order to formulate Fischer's
theorem: we denote by $\mathcal{P}\left( \mathbb{R}^{d}\right) $ the set of
all polynomials in the variable $x=\left( x_{1}, \dots ,x_{d}\right) \in 
\mathbb{R}^{d}$ with complex coefficients, by $\mathbb{N}_{0}$ the set of
natural numbers (emphasizing the fact that 0 is included), and by $\mathcal{P%
}_{m}\left( \mathbb{R}^{d}\right) $ the subspace of all homogeneous
polynomials of degree $m.$ Recall that a polynomial $P\left( x\right) $ is 
\emph{homogeneous} of degree $\alpha $ if $P\left( tx\right) =t^{\alpha
}P\left( x\right) $ for all $t>0$ and for all $x.$ Given a polynomial $%
P\left( x\right) $, we denote by $P^{\ast }\left( x\right) $ the polynomial
obtained from $P\left( x\right) $ by conjugating its coefficients, and by $%
P\left( D\right) $ be the linear differential operator obtained by replacing
the variable $x_{j}$ by the differential operator $\frac{\partial }{\partial
x_{j}}$. It is well known that a polynomial $P\left( x\right) $ of degree $k$
can be written as a sum of homogeneous polynomials $P_{j}\left( x\right) $
of degree $j$ for $j=0, \dots ,k,$ so 
\begin{equation*}
P\left( x\right) =P_{k}\left( x\right) +\cdots +P_{0}\left( x\right) ,
\end{equation*}%
and we call the homogeneous polynomial $P_{k}\left( x\right) $ the leading
term. Fischer's theorem states that given a \emph{homogeneous} polynomial $P$%
, the following decomposition holds: for each polynomial $f\in \mathcal{P}%
\left( \mathbb{R}^{d}\right) $ there exist \emph{unique} polynomials $q\in 
\mathcal{P}\left( \mathbb{R}^{d}\right) $ and $h\in \mathcal{P}\left( 
\mathbb{R}^{d}\right) $ such that%
\begin{equation*}
f=P\cdot q+h\text{ and }P^{\ast }\left( D\right) h=0.
\end{equation*}

We recall the standard notation for multi-indices $\alpha =\left( \alpha
_{1}, \dots ,\alpha _{d}\right) \in \mathbb{N}_{0}^{d}$: set $x^{\alpha
}=x_{1}^{\alpha _{1}} \dots x_{d}^{\alpha _{d}},$ $\alpha !=\alpha _{1}!
\dots \alpha _{d}!$, and $\left\vert \alpha \right\vert =\alpha _{1}+\cdots
+\alpha _{d}$. Let $P$ and $Q$ be given by%
\begin{equation*}
P\left( x\right) =\sum_{\alpha \in \mathbb{N}_{0}^{d},\left\vert \alpha
\right\vert \leq N}c_{\alpha }x^{\alpha }\text{ and }Q\left( x\right)
=\sum_{\alpha \in \mathbb{N}_{0}^{d},\left\vert \alpha \right\vert \leq
M}d_{\alpha }x^{\alpha },
\end{equation*}
where $c_\alpha, d_\alpha \in \mathbb{C}$.

An important ingredient in the proof of Fischer's Theorem is the \emph{%
Fischer inner product} $\left[ \cdot ,\cdot \right] _{F}$ on $\mathcal{P}%
\left( \mathbb{R}^{d}\right) $, defined by 
\begin{equation}
\left[ P,Q\right] _{F}:=\left( Q^*\left( D\right) P\right)
\;(0)=\sum_{\alpha \in \mathbb{N}_{0}^{d}}\alpha !c_{\alpha }\overline{%
d_{\alpha }}.  \label{eqQPD}
\end{equation}

The Fischer inner product has been used by many authors under different
names, see \cite{deRo92} and \cite{Saue00}, and the references therein. In 
\cite{NeSh66} and \cite{Shap89} the corresponding Hilbert space norm $\sqrt{%
\left[ P,P\right] _{F}}$ is called the \emph{Fischer norm}, while in \cite%
{Beau97} and \cite{Zeil} the term \emph{Bombieri norm} is used.

One aim in Shapiro's paper \cite{Shap89} is to provide Fischer
decompositions in a wider setting, going beyond the case of polynomials to
more general function spaces, in particular to the space $E\left( \mathbb{C}%
^{d}\right) $ of all entire functions $f:\mathbb{C}^{d}\rightarrow \mathbb{C}
$. It is convenient to adopt a notion introduced in \cite[p. 522]{Shap89};
suppose that $E$ is a vector space of infinitely differentiable functions $%
f:G\rightarrow \mathbb{C}$ (defined on an open subset $G$ in $\mathbb{R}^{d}$
or $\mathbb{C}^{d}$) that is a module over $\mathcal{P}\left( \mathbb{R}
^{d}\right) $: then we say that a polynomial $P$ and a differential operator 
$Q\left( D\right) $ form a \emph{Fischer pair for the space }$E$, if for
each $f\in E$ there exist \emph{unique} elements $q\in E$ and $h\in E$ such
that 
\begin{equation}
f=P\cdot q+h\text{ and }Q\left( D\right) h=0.  \label{eqDecomp}
\end{equation}%
Shapiro proved in \cite[Theorem 1]{Shap89} that Fischer's theorem is also
true when $\mathcal{P}\left( \mathbb{R}^{d}\right) $ is replaced by $E\left( 
\mathbb{C}^{d}\right) $. His approach is based on homogeneous expansions of
entire functions and estimates of the Fischer norms for homogeneous
polynomials.

Shapiro also raised the question whether other types of Fischer pairs could
be found. In \cite{Rend08} the first author identified new kinds of Fischer
pairs, related to the polyharmonic operator $\Delta ^{k}$, where $\Delta
^{k} $ is the $k$-th iterate of the Laplace operator 
\begin{equation*}
\Delta =\frac{\partial ^{2}}{\partial x_{1}^{2}}+ \cdots +\frac{\partial ^{2}%
}{\partial x_{d}^{2}}.
\end{equation*}%
It is shown in \cite{Rend08} that a polynomial $P\left( x\right) $ of degree 
$2k$ and the differential operator $Q\left( D\right) :=\Delta ^{k}$ form a
Fischer pair for $\mathcal{P}\left( \mathbb{R}^{d}\right) $ provided that
the leading term $P_{2k}$ is non-zero and \emph{\ non-negative} (i. e., $%
P_{2k}\ge 0$), thus guaranteeing Fischer decompositions for the pair $\left(
P,\Delta ^{k}\right) $ with respect to the vector space $\mathcal{P}\left( 
\mathbb{R}^{d}\right) .$ To pass from $\mathcal{P}\left( \mathbb{R}%
^{d}\right) $ to the space of entire functions $E\left( \mathbb{C}%
^{d}\right) $ is a non-trivial task which involves careful analysis: it is
shown in \cite{Rend08} that $\left( P,\Delta ^{k}\right) $ is Fischer pair
for $E\left( \mathbb{C}^{d}\right) $ if the leading polynomial $P_{2k}\left(
x\right) $ is elliptic, i.e., if there exists a constant $C>0$ such that 
\begin{equation*}
P_{2k}\left( x\right) \geq C\left\vert x\right\vert ^{2k}\text{ for all }%
x\in \mathbb{R}^{d}.
\end{equation*}
Thus a Fischer decomposition holds for \emph{entire} functions when $P_{2k}$
is elliptic.

In the present paper we want to discuss Fischer decompositions for entire
functions when we relax the assumption of ellipticity. We remark that while
the definition of ``Fischer pair'' entails the uniqueness of the
decomposition, we shall use \emph{Fischer decomposition} in a wider sense,
that includes the possibility of not having uniqueness.

Let us denote the unit sphere by 
\begin{equation*}
\mathbb{S}^{d-1} =\left\{ \theta \in \mathbb{R}^{d-1}:\left\vert \theta
\right\vert =1\right\},
\end{equation*}%
its surface area measure by $d\theta $, and its area by $\omega_{d - 1} = |%
\mathbb{S}^{d-1}|$. We define for $f,g\in $ $\mathcal{P}\left( \mathbb{R}%
^{d}\right) $ the inner product, together with its associated norm, 
\begin{equation}
\left\langle f,g\right\rangle _{L^2(\mathbb{S}^{d-1})}:=\int_{\mathbb{S}%
^{d-1}}f\left( \theta \right) \overline{g\left( \theta \right) }d\theta 
\text{ and }\left\Vert f\right\Vert _{L^2(\mathbb{S}^{d-1})}=\sqrt{%
\left\langle f,f\right\rangle _{L^2(\mathbb{S}^{d-1})}}.  \label{eqprodspher}
\end{equation}

Given a polynomial $P$, to avoid excessive subindices we shall denote by
both $P$ and $M_{P}$ the multiplication operator associated to $P$: for
every function $f$, we set $M_{P}(f):=Pf$.

With the convention that in the case $\alpha =0$, the expression $\left(
2k-\beta \right) /\alpha $ is to be interpreted as $\infty $, our first main
result states the following:

\begin{theorem}
\label{ThmMain1}Let $P_{2k}$ be a homogeneous polynomial of degree $2k>0$
such that there exist $C>0$, $D>0$ and $\alpha \geq 0$ with 
\begin{equation*}
\left\langle P_{2k}f_{m},f_{m}\right\rangle _{L^{2}(\mathbb{S}^{d-1})}\geq 
\frac{1}{C\left( m+D\right) ^{\alpha }}\left\langle f_{m},f_{m}\right\rangle
_{\mathbb{S}^{d-1}}
\end{equation*}%
for all homogeneous polynomials $f_{m}$ of degree $m.$ Let $0\leq \beta <2k$
and for $j=0,\dots ,\beta $, let the polynomials $P_{j}$ be homogeneous of
degree $j$. Then for every entire function $f$ of order $\rho <\left(
2k-\beta \right) /\alpha $, there exist entire functions $q$ and $r$ of
order $\leq \rho $ such that 
\begin{equation*}
f=\left( P_{2k}-P_{\beta }-\cdots -P_{0}\right) q+r\text{ and }\Delta
^{k}r=0.
\end{equation*}
\end{theorem}

\begin{proof}
By Proposition \ref{PropT} with $C_m := C\left( m+D\right)^{- \alpha }$, we
have that $\left( P_{2k},\Delta ^{k}\right) $ is a Fischer pair for $%
\mathcal{P}\left( \mathbb{\ R}^{d}\right) $ and furthermore, 
\begin{equation*}
\left\Vert Tf_{m}\right\Vert _{L^2(\mathbb{S}^{d-1})}\leq C\left( m+D\right)
^{\alpha }\left\Vert f_{m}\right\Vert _{L^2(\mathbb{S}^{d-1})}
\end{equation*}
for every homogeneous polynomial $f_m$ of degree $m$. Now the result follows
from Theorem \ref{Thm6}.
\end{proof}

First, let us see how Theorem \ref{ThmMain1} together with Theorem \ref%
{Thm6b} imply the following result of D. Armitage, cf. \cite[Theorem 1]%
{Armi04}, which is a refinement of a result of D. Khavinson and H. Shapiro,
cf. \cite[Theorem 1]{KhSh92}. Actually, the Theorem of Armitage is improved,
since there we have that if $0<\rho \left( h\right) =\rho \left( f\right) <
\infty$ then $\tau \left( h\right) \leq C_{1}\tau \left( f\right) $, where $%
C_{1}$ depends only on the domain $\Omega$, and in the case where $\Omega $
is a ball, $C_{1}=1.$ It follows from our results that $C_1$ can be taken to
be 1 for arbitrary ellipsoids, and not just the ball.

\begin{theorem}
Le $\Omega =\left\{ \left( x_{1},\dots ,x_{d}\right) \in \mathbb{R}^{d}:%
\frac{x_{1}^{2}}{a_{1}^{2}}+\cdots +\frac{x_{d}^{2}}{a_{d}^{2}}<1\right\} $
be an ellipsoid, where the constants $a_{1},\dots ,a_{d}$ are assumed to be
positive numbers. Then for every entire function $f$ on $\mathbb{C}^{d}$,
the solution $h$ of the Dirichlet problem for $\Omega $ with data function $%
f $ (restricted to the boundary) has a harmonic continuation to $\mathbb{R}%
^{d} $ and hence a continuation to an entire function on $\mathbb{C}^{d}$.
Furthermore, denoting also by $h$ the said extension, we have $\rho \left(
h\right) \leq \rho \left( f\right)$, and if $0<\rho \left( h\right) =\rho
\left( f\right) < \infty $, then $\tau \left( h\right) \leq \tau \left(
f\right) $.
\end{theorem}

\begin{proof}
In order to apply Theorem \ref{ThmMain1} we take $k=1$, $P_{2}\left(
x\right) =\frac{x_{1}^{2}}{a_{1}^{2}}+\cdots +\frac{x_{d}^{2}}{a_{d}^{2}},$ $%
P_{0}=1$ and $\beta =0.$ Set $C = \max\{a_1^2, \dots, a_d^2\}$ and note that
the restriction of $P_2$ to $\mathbb{S}^{d-1}$ satisfies $P_2 \ge 1/ C$.
Thus, the integral inequality from Theorem \ref{ThmMain1} with $\alpha =0$
follows: 
\begin{equation*}
\left\langle P_{2}f_{m},f_{m}\right\rangle _{L^{2}(\mathbb{S}^{d-1})}\geq 
\frac{1}{C}\left\langle f_{m},f_{m}\right\rangle _{L^{2}(\mathbb{S}^{d-1})},
\end{equation*}
for all homogeneous polynomials $f_{m}$ of degree $m.$ Then $\left( 2k-\beta
\right) /\alpha =\infty $ and Theorem \ref{ThmMain1} says that for any
entire function $f$ of order $\rho <\infty $ there exist entire functions $q$
and $h$ of order $\leq \rho $ such that $f=\left( P_{2}-1\right) q+h$ and $%
\Delta h=0.$ Thus $h$ is an entire harmonic function with $f\left( \xi
\right) =h\left( \xi \right) $ for all $\xi \in \partial \Omega .$ By
Theorem \ref{Thm6b} the function $q$ has either order $<\rho ,$ or order $%
\rho $ and type $\leq \tau \left( f\right) .$ It follows that if $h$ has
order exactly $\rho$, its type must satisfy $\tau \left( h \right)\leq \tau
\left( f\right) $ if $h$ has order $\rho .$
\end{proof}

The results in the present paper can also be used to deal with the Dirichlet
problem for parabolas, including degenerate cases such as the strip, with
boundary given by $\left( x_{2}+a\right) (x_{2}-a)=0$. For an arbitrary
nondegenerate parabola, after a translation and a rotation we may assume
that it is symmetric with respect to the $x$-axis and has the origin as its
vertex, so it is defined by the equation $ax_{1}=x_{2}^{2}$.

In order to apply Theorem \ref{ThmMain1} to these examples the assumed
integral inequality needs to be proven. We do so for $d=2$ and $%
P_{2}(x_{1},x_{2})=x_{2}^{2}$ in the second main result. Higher dimensional
generalizations will be pursued elsewhere.

\begin{theorem}
\label{ThmMainest} Let $d=2$. Then for all homogeneous polynomials $f_{m}$
of degree $m$ the following inequality holds: 
\begin{equation}
\left\langle x_{2}^{2}f_{m},f_{m}\right\rangle _{L^{2}(\mathbb{S}^{1})}\geq 
\frac{\pi ^{2}}{4\left( m+4\right) ^{2}}\left\langle
f_{m},f_{m}\right\rangle _{L^{2}(\mathbb{S}^{1})}.  \label{eqmain2}
\end{equation}
\end{theorem}

\begin{theorem}
\label{Thmparabolas}Let $d=2$, let $f(x_{1},x_{2})$ be an entire function of
order $\rho \left( f\right) <\frac{1}{2}$, and let $ax_{1}=x_{2}^{2}$ define
the locus of a parabola. Then there exists an entire harmonic function $h$
of order $\leq \rho $ such that $h=f$ on $\Omega :=\left\{ \left(
x_{1},x_{2}\right) \in \mathbb{R}^{2}:ax_{1}=x_{2}^{2}\right\} .$
\end{theorem}

\begin{proof}
Set $P_{2}(x_{1},x_{2})=x_{2}^{2}$ (so $k=1$), $P_{1}(x_{1},x_{2})=ax_{1}$,
and $P_{0}(x_{1},x_{2})=0$. Since $\alpha = 2$ in Theorem \ref{ThmMainest},
it follows that $\left( 2k-\beta \right) /\alpha =\frac{1}{2}.$ By Theorems %
\ref{ThmMainest} and \ref{ThmMain1} there exist entire functions $q$ and $h$
of order at most $\rho < 1/2$, such that $h$ is harmonic and $%
f(x_{1},x_{2})=\left( x_{2}^{2}-ax_{1}\right) q(x_{1},x_{2})+h(x_{1},x_{2}).$
Thus, $f=h$ on $\Omega $.
\end{proof}

\begin{theorem}
Let $d=2$, let $f(x_{1},x_{2})$ be an entire function of order $\rho \left(
f\right) <1$, and consider the strip with locus defined by $\left(
x_{1}-a\right) \left( x_{1}+a\right) =0$, where $a > 0$. Then there exists
an entire harmonic function $h$ of order $\leq \rho $ such that $h=f$ on $%
\Omega :=\left\{ \left( x_{1},x_{2}\right) \in \mathbb{R}^{2}:x_{1}^{2} =
a^{2}\right\} .$
\end{theorem}

\begin{proof}
Set $P_{2}(x_{1},x_{2})=x_{1}^{2}$ (so $k=1$), $P_{1}(x_{1},x_{2})=0$, and $%
P_{0}(x_{1},x_{2})=a^{2}$. Then $\left( 2k-\beta \right) /\alpha =1.$ By
Theorems \ref{ThmMainest} and \ref{ThmMain1} there exist entire functions $q$
and $h$ of order at most $\rho $, such that $h$ is harmonic and $%
f(x_{1},x_{2})=\left(x_{1}^{2}-a^{2}\right) q(x_{1},x_{2})+h(x_{1},x_{2}).$
Thus, $f=h$ on $\Omega $.
\end{proof}

The Dirichlet problem on the strip in $\mathbb{R}^{2}$ was discussed in the
classical paper \cite{Widd60}, see also \cite{Dura03}. For the Dirichlet
problem on the slab $S_{a,b}:=\left( a,b\right) \times \mathbb{R}^{d-1}$
(which is a strip for $d=2)$ we mention the following result in \cite{KLR17b}%
: each entire function $f$ has a decomposition 
\begin{equation*}
f\left( x_{1},\dots ,x_{d}\right) =\left( x_{1}-a\right) \left(
x_{1}-b\right) q\left( x_{1},\dots ,x_{d}\right) +h\left( x_{1},\dots
,x_{d}\right)
\end{equation*}%
where $q$ is an entire function and $h$ is entire and harmonic. Thus the
existence of a Fischer decomposition is proved, but uniqueness of the
representation is lost. For example, the function $f\left(
x_{1},x_{2}\right) =\sin \left( \frac{\pi }{a}x_{1}\right) e^{\frac{\pi }{ a 
}x_{2}}$ is harmonic on $\mathbb{R}^{2}$ and it vanishes on the boundary of
the strip. Hence, the function $f$ has at least two decompositions, the
first where $q=0$ and $h=f,$ the second where $h=0$ and 
\begin{equation*}
q=\frac{\sin \left( \frac{\pi }{a}x_{1}\right) e^{\frac{\pi }{a }x_{2}}}{%
\left( x_{1}-a\right) \left( x_{1}+a\right) },
\end{equation*}%
which is an entire function.

Another interesting example is the ellipsoidal cylinder 
\begin{equation*}
\Omega _{\text{cyl}}=\left\{ x=\left( x_{1},\dots ,x_{d}\right) \in \mathbb{R%
}^{d}:\frac{x_{1}^{2}}{a_{1}^{2}}+\cdots +\frac{x_{d-1}^{2}}{a_{d-1}^{2}}%
<1\right\}
\end{equation*}%
for given positive numbers $a_{1}, \dots ,a_{d-1}$ and $d\geq 2.$ One can
use the results in this paper to conclude that for each entire function $f$
of order $\rho <1$, there exist entire functions $q$ and $h$ of order
bounded by $\rho$, such that 
\begin{equation*}
f=\left( \frac{x_{1}^{2}}{a_{1}^{2}}+\cdots +\frac{x_{d-1}^{2}}{a_{d-1}^{2}}%
-1\right) q+h\text{ and }\Delta h=0.
\end{equation*}%
Thus $h$ is an entire harmonic function of order $\leq \rho $ solving the
Dirchlet problem for the function $f$ and the cylinder -- a result which was
proven by a different method in \cite{KLR17}. It is an open question whether
for \emph{any} entire data function $f$ (restricted to the boundary of a
cylinder) there exists a harmonic entire function $h$ that solves the
Dirichlet problem for $f$ and the cylinder. Results in \cite{GaRe16} and 
\cite{GaRe19} for extending harmonic functions vanishing on the boundary of
the cylinder indicate that a positive answer is possible.

Finally let us mention that the Dirichlet problem for the halfspace is
discussed in \cite{Gard81}, see also \cite{Mady21}. For the Dirichlet
problem for general unbounded domains we refer to \cite{Gard93}.

\section{Entire functions}

A point in $\mathbb{C}^{d}$ is denoted by $z=\left( z_{1},\dots
,z_{d}\right) $, and by $\left\vert z\right\vert =\sqrt{ \left\vert
z_{1}\right\vert ^{2}+ \cdots +\left\vert z_{d}\right\vert ^{2}}$ its
Euclidean norm. For a continuous function $f:\mathbb{C}^{d}\rightarrow 
\mathbb{C}$ we define 
\begin{equation*}
M_{\mathbb{C}^{d}}\left( f,r\right) :=\sup \left\{ \left\vert f\left(
z\right) \right\vert :z\in \mathbb{C}^{d},\left\vert z\right\vert =r\right\},
\label{eqM1}
\end{equation*}%
and then the order $\rho _{\mathbb{C}^{d}}\left( f\right) $ of $f$ is 
\begin{equation*}
\rho _{\mathbb{C}^{d}}\left( f\right) =\lim_{r\rightarrow \infty }\sup \frac{%
\log \log M_{\mathbb{C}^{d}}\left( f,r\right) }{\log r}\in \left[ 0,\infty %
\right] .
\end{equation*}%
If $0<\rho _{\mathbb{C}^{d}}\left( f\right) <\infty $ then the type of $f$
is given by 
\begin{equation*}
\tau _{\mathbb{C}^{d}}\left( f\right) =\lim_{r\rightarrow \infty }\sup \frac{%
\log M_{\mathbb{C}^{d}}\left( f,r\right) }{r^{\rho \left( f\right) }}.
\end{equation*}%
There is a vast literature about entire analytic functions $f:\mathbb{C}%
^{d}\rightarrow \mathbb{C}$ of finite order. In this paper we will consider
also the order of a harmonic function $f:\mathbb{R}^{d}\rightarrow \mathbb{R}
$, which is defined in terms of real variables, see \cite{Frya78}, \cite%
{FrSh87}, \cite{Fuga80} or \cite{Armi04} for details, or the exposition
below.

For our purposes it is necessary to introduce now some terminology and
notations: let $B_{R}:=\left\{ x\in \mathbb{R}^{d}:\left\vert x\right\vert
<R\right\} $ be the open ball in $\mathbb{R}^{d}$ with center $0$ and radius 
$0<R\leq \infty .$ Assume that $f$ is an infinitely differentiable function
on $B_{R}.$ Define a homogeneous polynomial of degree $m$ by 
\begin{equation}  \label{taylor}
f_{m}\left( x\right) =\sum_{\left\vert \alpha \right\vert =m}\frac{1}{\alpha
!}\frac{\partial ^{\alpha }f}{\partial x^{\alpha }}\left( 0\right)
\;x^{\alpha }\text{ for }m\in \mathbb{N}_{0}.
\end{equation}
Let $A\left( B_{R}\right) $ be the set of all infinitely differentiable
functions $f:B_{R}\rightarrow \mathbb{C}$ such that for every compact subset 
$K\subset B_{R}$, the homogeneous Taylor series $\sum_{m=0}^{\infty
}f_{m}\left( x\right) $ converges absolutely and uniformly to $f$ on $K.$

For $x\in \mathbb{R}^{d}$ define $r:=\left\vert x\right\vert $ and $\theta
=x/\left\vert x\right\vert ,$ so $\theta \in \mathbb{S}^{d-1}$ and $%
x=r\theta .$ Then 
\begin{equation}  \label{extfun}
f\left( x\right) =\sum_{m=0}^{\infty }f_{m}\left( x\right)
=\sum_{m=0}^{\infty }f_{m}\left( r\theta \right) =\sum_{m=0}^{\infty
}r^{m}f_{m}\left( \theta \right) =f\left( r\theta \right),
\end{equation}%
which can be seen as a power series in the real variable $r$ and the
coefficients $f_{m}\left( \theta \right) $, where $\theta \in \mathbb{S}%
^{d-1}$. Clearly we can replace $r$ by a complex variable $\zeta $,
defining, for $z=\zeta \theta :=\left( \zeta \theta _{1},\dots ,\zeta \theta
_{d}\right) $ and $\theta =\left( \theta _{1},\dots ,\theta _{d}\right) \in 
\mathbb{S}^{d-1}$, the function 
\begin{equation*}
f\left( \zeta \theta \right) =\sum_{m=0}^{\infty }\zeta ^{m}f_{m}\left(
\theta \right),
\end{equation*}%
which is entire in $\zeta \in \mathbb{C}$ for any fixed $\theta \in \mathbb{S%
}^{d-1}$. One can also replace $x=\left( x_{1},\dots ,x_{d}\right) $ in
formula (\ref{extfun}) by the complex vector $z=\left( z_{1},\dots
.,z_{d}\right) \in \mathbb{C}^{d},$ and the convergence of the sum $f\left(
z\right) $ follows from a result due to J. Siciak in \cite{Sici74}, stating
that for each homogeneous polynomial $f_{m}\left( z\right) $, the following
estimate holds:%
\begin{equation*}
\sup_{z\in \mathbb{C}^{d},\left\vert z\right\vert \leq 1}\left\vert
f_{m}\left( z\right) \right\vert \leq \sqrt{2}\sup_{\zeta \in \mathbb{C}%
,\left\vert \zeta \right\vert \leq 1}\sup_{\theta \in \mathbb{S}%
^{d-1}}\left\vert f_{m}\left( \zeta \theta \right) \right\vert =\sqrt{2}%
\sup_{\theta \in \mathbb{S}^{d-1}}\left\vert f_{m}\left( \theta \right)
\right\vert .
\end{equation*}%
Thus each function $f\in A\left( B_{R}\right) $ with $R=\infty $ has an
extension to an entire function $f:\mathbb{C}^{d}\rightarrow \mathbb{C}.$ In
general it is known that $A\left( B_{R}\right) $ is isomorphic to the set of
all holomorphic functions on the harmonicity hull of $B_{R},$ see e.g. \cite%
{Rend08} for more details.

In analogy to (\ref{eqM1}), one defines 
\begin{equation*}
M_{\mathbb{R}^{d}}\left( f,r\right) :=\sup \left\{ \left\vert f\left( \zeta
\theta \right) \right\vert :\zeta \in \mathbb{C},\left\vert \zeta
\right\vert =r,\theta \in \mathbb{S}^{d-1}\right\},
\end{equation*}%
and the corresponding order $\rho _{\mathbb{R}^{d}}\left( f\right) $ of a
non-constant function $f\in A\left( B_{\infty }\right) $: 
\begin{equation*}
\rho _{\mathbb{R}^{d}}\left( f\right) =\lim_{r\rightarrow \infty }\sup \frac{%
\log \log M_{\mathbb{R}^{d}}\left( f,r\right) }{\log r}.
\end{equation*}%
For an entire function $f$ one has the estimate 
\begin{equation*}
M_{\mathbb{R}^{d}}\left( f,r\right) \leq M_{\mathbb{C}^{d}}\left( f,r\right)
\leq M_{\mathbb{R}^{d}}\left( f,\sqrt{2}r\right),
\end{equation*}%
which leads to the statements 
\begin{equation*}
\rho _{\mathbb{C}^{d}}\left( f\right) =\rho _{\mathbb{R}^{d}}\left( f\right) 
\text{ and }\tau _{\mathbb{C}^{d}}\left( f\right) =\sqrt{2} \ \tau _{\mathbb{%
R}^{d}}\left( f\right),
\end{equation*}%
where $\tau _{\mathbb{R}^{d}}\left( f\right) $ is the type with respect to $%
M_{\mathbb{R}^{d}}\left( f,r\right),$ defined as 
\begin{equation*}
\tau _{\mathbb{R}^{d}}\left( f\right) =\lim_{r\rightarrow \infty }\sup \frac{
\log M_{\mathbb{R}^{d}}\left( f,r\right) }{ r^\rho},
\end{equation*}
provided $0 < \rho := \rho _{\mathbb{R}^{d}} < \infty $. In this paper we
shall work with $\rho _{\mathbb{R}^{d}}\left( f\right) $ and $\tau _{\mathbb{%
R}^{d}}\left( f\right) $, which are more natural in the setting of real
euclidean spaces.

It is not difficult to prove that for $f\in A\left( B_{\infty }\right) $ one
has 
\begin{equation*}
\rho _{\mathbb{R}^{d}}\left( f\right) =\lim_{m\rightarrow \infty }\sup \frac{%
\log m}{\log \left( \sqrt[m]{\frac{1}{\max_{\theta \in \mathbb{S}%
^{d-1}}\left\vert f_{m}\left( \theta \right) \right\vert }}\right) }.
\end{equation*}%
Frequently we shall use the following formulation: let $f\in A\left(
B_{\infty }\right) $ and let $\rho \ge 0$. Then $\rho _{\mathbb{R}%
^{d}}\left( f\right) \leq \rho < \infty$ if and only if for every $%
\varepsilon >0$ there exists an $m_0 \geq 0$ such that for every $m \ge m_0$
we have 
\begin{equation}
\max_{\theta \in \mathbb{S}^{d-1}}\left\vert f_{m}\left( \theta \right)
\right\vert \leq \frac{1}{m^{m/\left( \rho +\varepsilon \right) }}.
\label{eqtaylor}
\end{equation}%
Let $f$ be an entire function of order $0<\rho <\infty $ and type $\tau$.
According to a theorem of Lindel\"{o}f and Pringsheim, 
\begin{equation*}
\lim_{m\rightarrow \infty }\sup \left( m\cdot \max_{\theta \in \mathbb{S}%
^{d-1}}\left\vert f_{m}\left( \theta \right) \right\vert ^{\frac{\rho }{m}%
}\right) =e\rho \tau .
\end{equation*}%
Thus for every $\varepsilon >0$ there exists $m_{0}\in \mathbb{N}$ such that
for all $m\geq m_{0}$%
\begin{equation}  \label{LindPring}
\max_{\theta \in \mathbb{S}^{d-1}}\left\vert f_{m}\left( \theta \right)
\right\vert \leq \frac{\left( e\rho \tau +\varepsilon \right) ^{m/\rho }}{%
m^{m/\rho }}.
\end{equation}

On the other hand, if we know that $f$ is entire and there exists an $m_0$
such that for all $m \ge m_0$ inequality (\ref{LindPring}) is satisfied (for
some constants $\rho$ and $\tau$ and every $\varepsilon >0$) standard
arguments show that the order of $f$ is at most $\rho$, and if it is equal
to $\rho$, then its type is at most $\tau$. An analogous remark can be made
regarding inequality (\ref{eqtaylor}).

It can be proved that the sum (\ref{extfun}) also converges if we replace $%
x\in \mathbb{R}^{d}$ by a complex vector $z\in \mathbb{C}^{d}$ in the
polynomials $f_{m}\left( x\right) .$

The next result can be found in \cite[Theorem 11]{Rend08} (save for a
normalization error: $\sqrt{\omega _{d-1}}$ should appear in the
denominator, as it does below).

\begin{theorem}
\label{thm:norm} For all homogeneous polynomials $f_{m}$ of degree $m\in 
\mathbb{N}_{0}$ we have 
\begin{equation*}
\max_{\theta \in \mathbb{S}^{d-1}}\left\vert f_{m}\left( \theta \right)
\right\vert \leq \frac{\sqrt{2}}{\sqrt{\omega _{d-1}}}\left( 1+m\right)
^{\left( d-1\right) /2}\left\Vert f_{m}\right\Vert _{L^2(\mathbb{S}^{d-1})}.
\end{equation*}
\end{theorem}

\section{Estimates for the Fischer decomposition}

Assume that $\left( P, Q \right) $ is a Fischer pair for the vector space $%
\mathcal{P}\left( \mathbb{R}^{d}\right) $. By definition, for each
polynomial $f $ there exist unique polynomials $q$ and $r$ such that 
\begin{equation}
f=Pq+r\text{ with } Q(D) r=0.  \label{eqdede}
\end{equation}%
Since the decomposition in (\ref{eqdede}) is unique we can define operators $%
T_{P}:\mathcal{P}\left( \mathbb{R}^{d}\right) \rightarrow \mathcal{P}\left( 
\mathbb{R}^{d}\right) $ and $R_{P}:\mathcal{P}\left( \mathbb{R}^{d}\right)
\rightarrow \mathcal{P}\left( \mathbb{R}^{d}\right) $ by setting $%
T_{P}\left( f\right) :=q$ and $R_{P}\left( f\right) :=r.$ So we write (as in 
\cite{Rend08} or \cite{KhSh92}) 
\begin{equation}
f=P\cdot T_{P}\left( f\right) +R_{P}\left( f\right) .  \label{defTT}
\end{equation}%
It is easy to see that $T_{P}$ and $R_{P}$ are \emph{linear} operators. Now
assume that $P$ is a polynomial of degree $2k$, and write 
\begin{equation}
P=P_{2k}-P_{2k-1}- \cdots -P_{0},\   \label{eqHom}
\end{equation}%
where the $P_{j}$ are homogeneous polynomials for $j=0,\dots ,2k$. The minus
signs are chosen to have a simplified expression in the next Theorem, which
describes the operator $T_{P}$ using just the multiplication operators $%
P_{j} $ and the operator $T:=T_{P_{2k}}$ for the leading term $P_{2k}.$

\begin{theorem}
Let $Q$ be a homogeneous polynomial of degree $2k > 0$, let $P$ be a
polynomial of degree $2k$ of the form (\ref{eqHom}), and assume that $%
(P_{2k},Q)$ is a Fischer pair for $\mathcal{P}\left( \mathbb{R}^{d}\right) $%
. Setting $T:=T_{P_{2k}}$, we have 
\begin{equation}
T_{P}\left( f_{m}\right)
=\sum_{j=-1}^{m}\sum_{s_{0}=0}^{2k-1}\sum_{s_{1}=0}^{2k-1} \dots
\sum_{s_{j}=0}^{2k-1}T(P_{s_{j}}T(\cdots P_{s_{1}}
T(P_{s_{0}}T(f_{m}))\cdots))  \label{eqinduct}
\end{equation}
for all homogeneous polynomials $f_{m}$ of degree $m$, with the convention
that the summand for $j=-1$ is $Tf_{m}$.
\end{theorem}

\begin{proof}
We will simplify the above expression by omitting parentheses. The proof
follows by induction over the degree $m.$ First we formulate the statement
in a slightly different form:

(i) For each homogeneous polynomial $f_{m}$ of degree $m$ there exists a
decomposition 
\begin{equation*}
f_{m}=Pq_{m}+r_{m},
\end{equation*}
where $q_{m}:= T_P (f_m)$ and $r_{m}$ is a polynomial of degree $\leq m$
satisfying $Q\left( D\right) r_{m}=0$.

For $m=0$ the polynomial $f_{0}$ is constant and $Q\left( D\right) f_{0}=0$,
since $Q$ is homogeneous of degree $2k>0.$ So we have the decomposition $%
f_{0}=P\cdot q_{0}+f_{0}$ with $q_{0}=0.$ On the other hand, for $m=0$ the
right hand side has summands for $j=-1$ and $j=0,$ namely $Tf_{0}$ and $%
\sum_{s_{0}=0}^{2k-1}TP_{s_{0}}Tf_{0}.$ Now use that $Tf_{0}=0.$

Assume that the statement holds for all homogeneous polynomials of degree $%
\leq m.$ Let $f_{m+1}$ be a homogeneous polynomial of degree $m+1.$ Since $%
(P_{2k},Q)$ is a Fischer pair we can write 
\begin{equation*}
f_{m+1}=P_{2k}\cdot Tf_{m+1}+r_{m+1}
\end{equation*}
where $Q\left( D\right) r_{m+1}=0.$ Then 
\begin{equation*}
f_{m+1}=P\cdot Tf_{m+1}+\sum_{s_{m + 1}=0}^{2k-1}P_{s_{m +
1}}Tf_{m+1}+r_{m+1}.
\end{equation*}%
For $0 \le s_{m + 1} \le k - 1$ define $g_{s_{m + 1}}=P_{s_{m + 1}}Tf_{m+1}$%
, which is a homogeneous polynomial of degree $\leq m.$ By the induction
hypothesis we can write 
\begin{equation*}
g_{s_{m + 1}}=P\cdot q_{s_{m + 1}}+r_{s_{m + 1}},
\end{equation*}
where $Q\left( D\right) r_{s_{m + 1}}=0$ and $q_{s_{m + 1}}$ is given by 
\begin{equation*}
q_{s_{m +
1}}=\sum_{j=-1}^{m}\sum_{s_{0}=0}^{2k-1}\sum_{s_{1}=0}^{2k-1}\cdots
\sum_{s_{j}=0}^{2k-1}TP_{s_{j}} \cdots TP_{s_{0}}Tg_{s_{m + 1}}.
\end{equation*}
It follows that 
\begin{equation*}
f_{m+1}=P\left( Tf_{m+1}+\sum_{s_{m + 1}=0}^{2k-1}q_{s_{m + 1}}\right)
+\sum_{s_{m + 1}=0}^{2k-1}r_{s_{m + 1}}+r_{m+1}.
\end{equation*}
This shows that we have a decomposition of $f_{m+1}$ of the desired form.
Using the induction hypothesis we see that $q_{m+1}:=Tf_{m+1}+ \sum_{s_{m +
1}=0}^{2k-1}q_{s_{m + 1}}$, so 
\begin{eqnarray*}
q_{m+1} &=&Tf_{m+1}+\sum_{s_{m +
1}=0}^{2k-1}\sum_{j=-1}^{m}\sum_{s_{0}=0}^{2k-1} \sum_{s_{1}=0}^{2k-1}\cdots
\sum_{s_{j}=0}^{2k-1}TP_{s_{j}} \cdots TP_{s_{0}}T\left( P_{s_{m +
1}}Tf_{m+1}\right) \\
&=&\sum_{j=-1}^{m+1}\sum_{s_{0}=0}^{2k-1}\sum_{s_{1}=0}^{2k-1} \cdots
\sum_{s_{j}=0}^{2k-1}TP_{s_{j}} \cdots TP_{s_{0}}Tf_{m+1}.
\end{eqnarray*}
\end{proof}

For the special case $P=P_{2k}-P_{0}$ we have contributions only when $%
s_{0}= $ $s_{1}=\cdots =s_{j}=0,$ so 
\begin{equation*}
T_{P}\left( f_{m}\right) =\sum_{j=-1}^{m}\left( T\circ M_{P_{0}}\right)
^{j+1}\circ \left( Tf_{m}\right) .
\end{equation*}

Given a real number $a$, set $a^+ := \max\{a, 0\}$.

\begin{proposition}
\label{boundssphere} Let $P_{2k}$ and $Q$ be homogeneous polynomials in $d$
variables, of degree $2k>0$, and assume that $(P_{2k},Q)$ is a Fischer pair
for $\mathcal{P}\left( \mathbb{R}^{d}\right) $. Write $T:=T_{P_{2k}}.$
Suppose there exist $C>0$, $D>0$ and $\alpha \geq 0$ such that 
\begin{equation}
\left\Vert Tf_{m}\right\Vert _{L^{2}(\mathbb{S}^{d-1})}\leq C\left(
m+D\right) ^{\alpha }\left\Vert f_{m}\right\Vert _{L^{2}(\mathbb{S}^{d-1})}
\label{eqTestimate}
\end{equation}%
for every homogeneous polynomial $f_{m}\left( x\right) $ of degree $m.$ Let
the polynomials $P_{j}\left( x\right) $ be homogeneous of degree $j$ for $%
j=0,\dots ,2k$, and define 
\begin{equation*}
D_{s}=\max_{\theta \in \mathbb{S}^{d-1}}\left\vert P_{s}\left( \theta
\right) \right\vert \text{ for }s=0,\dots ,2k-1.
\end{equation*}%
If $\beta <2k$ is a natural number such that $P_{j}=0$ whenever $j=\beta
+1,\dots ,2k-1$, then the following estimate holds for $s_{0},\dots
,s_{j}\in \left\{ 0,\dots ,2k-1\right\} $ and every homogeneous polynomial $%
f_{m}\left( x\right) $ of degree $m$: 
\begin{equation*}
\ \left\Vert TP_{s_{j}}\dots .TP_{s_{0}}Tf_{m}\right\Vert _{L^{2}(\mathbb{S}%
^{d-1})}\leq \left( m+D\right) ^{\alpha (j+1)}C^{j+1}D_{s_{j}}\cdots
D_{s_{0}}\left\Vert f_{m}\right\Vert _{L^{2}(\mathbb{S}^{d-1})}
\end{equation*}%
\begin{equation*}
\leq \left( m+D\right) ^{\frac{m\alpha }{2k-\beta }}C^{j+1}D_{s_{j}}\cdots
D_{s_{0}}\left\Vert f_{m}\right\Vert _{L^{2}(\mathbb{S}^{d-1})}.\ 
\end{equation*}
\end{proposition}

\begin{proof}
Let us write $C_{m}=C\left( m+D\right) ^{\alpha }.$ Note that $%
P_{s_{j}}T\dots P_{s_{0}}Tf_{m}$ is a homogeneous polynomial of degree 
\begin{equation}
d\left( s_{0},\dots ,s_{j}\right) :=\deg P_{s_{j}}T\dots .P_{s_{0}}Tf_{m}=
[m+\left( s_{0}-2k\right) +\dots +\left( s_{j}-2k\right)]^+ <m.
\label{eqdnull}
\end{equation}%
If $d\left( s_{0},\dots ,s_{j}\right) -2k$ is negative, this means $%
P_{s_{j}}T\dots .P_{s_{0}}Tf_{m}$ has degree $<2k,$ hence we see that 
\begin{equation*}
T\left[ P_{s_{j}}T\dots P_{s_{0}}Tf_{m}\right] =0.
\end{equation*}%
If $d\left( s_{0},\dots ,s_{j}\right) -2k\geq 0$, using $P_{s_{j}}=0$ when $%
s_{j}>\beta $, from (\ref{eqdnull}) get the following estimate for $j$: 
\begin{equation*}
j+1\leq j + 2 \le \frac{m+s_{0}+\dots +s_{j}}{2k}\leq \frac{m+\beta \left(
j+1\right) }{2k}.
\end{equation*}%
This is equivalent to 
\begin{equation*}
j+1\leq \frac{m}{2k-\beta }.
\end{equation*}

Now use the estimate (\ref{eqTestimate}) to see that 
\begin{eqnarray*}
\left\Vert TP_{s_{j}}\dots TP_{s_{0}}Tf_{m}\right\Vert _{L^2(\mathbb{S}%
^{d-1})} &\leq &C_{d_{j}}\left\Vert P_{s_{j}}TP_{s_{j-1}}\dots
TP_{s_{0}}Tf_{m}\right\Vert _{L^2(\mathbb{S}^{d-1})} \\
&\leq &C_{d_{j}}D_{s_{j}}\left\Vert TP_{s_{j-1}}\dots
TP_{s_{0}}Tf_{m}\right\Vert _{L^2(\mathbb{S}^{d-1})}.
\end{eqnarray*}%
Note that $C_{d_{j}}\leq C_{m}$. We can iterate the process, obtaining 
\begin{eqnarray*}
\left\Vert TP_{s_{j}}\dots TP_{s_{0}}Tf_{m}\right\Vert _{L^2(\mathbb{S}%
^{d-1})} &\leq &C_{m}^{j+1}\cdot D_{s_{j}}\cdots D_{s_{0}}\left\Vert
f_{m}\right\Vert _{L^2(\mathbb{S}^{d-1})} \\
&\leq &\left( m+D\right) ^{\frac{m\alpha }{2k-\beta }}C^{j+1}D_{s_{j}}\cdots
D_{s_{0}}\left\Vert f_{m}\right\Vert _{L^2(\mathbb{S}^{d-1})}.
\end{eqnarray*}
\end{proof}

The main difference between the next result and Theorem \ref{ThmMain1} is
that here we assume $\left( P,\Delta ^{k}\right) $ is a Fischer pair for $%
\mathcal{P}\left( \mathbb{R}^{d}\right) $.

\begin{theorem}
\label{Thm6} Let $P_{2k}$ be a homogeneous polynomial in $d$ variables, of
degree $2k>0$, and let $P$ be a polynomial of degree $2k$, having the form (%
\ref{eqHom}). Assume that $\left( P,\Delta ^{k}\right) $ is a Fischer pair
for $\mathcal{P}\left( \mathbb{R}^{d}\right) $. Write $T:=T_{2k}.$ Suppose
there exist $C>0$, $D>0$ and $\alpha \geq 0$ such that 
\begin{equation*}
\left\Vert Tf_{m}\right\Vert _{L^{2}(\mathbb{S}^{d-1})}\leq C\left(
m+D\right) ^{\alpha }\left\Vert f_{m}\right\Vert _{L^{2}(\mathbb{S}^{d-1})}
\end{equation*}%
for every homogeneous polynomial $f_{m}$ of degree $m$, and let $\beta <2k$
be a natural number such that $P_{j}=0$ whenever $j=\beta +1,\dots ,2k-1$.
Then for every entire function $f$ of order $\rho <\left( 2k-\beta \right)
/\alpha $, there are entire functions $q$ and $h$ of order $\leq \rho $,
satisfying 
\begin{equation}
f=\left( P_{2k}-P_{\beta }-\cdots -P_{0}\right) q +h\text{ and }\Delta
^{k}h=0.  \label{eqdecomp11}
\end{equation}
\end{theorem}

\begin{proof}
Assume $f$ is an entire function $f$ of order $<\left( 2k-\beta \right)
/\alpha $ and write $f=\sum_{m=0}^{\infty }f_{m}$, where the homogeneous
polynomials $f_{m}$ are given by (\ref{taylor}). Then $T_{P}\left(
f_{m}\right)$ is either the zero polynomial or a polynomial of degree $<m$
(not necessarily homogeneous). Our strategy is to show that 
\begin{equation}
q :=\sum_{m=0}^{\infty }T_{P}\left( f_{m}\right)  \label{eqgconvergence}
\end{equation}%
converges absolutely and uniformly on compacta, and hence it defines an
entire function $q:\mathbb{R}^{d}\rightarrow \mathbb{R}.$

Next, we determine the order of $q$ by writing $q\left( x\right)
=\sum_{M=0}^{\infty }G_{M}\left( x\right) $, where each $G_{M}$ is a
homogeneous polynomial of degree $M$ (as a reminder, we mention that when it
exists this representation is unique). Recall that 
\begin{equation}
T_{P}\left( f_{m}\right)
=\sum_{j=-1}^{m}\sum_{s_{0}=0}^{2k-1}\sum_{s_{1}=0}^{2k-1}\dots
\sum_{s_{j}=0}^{2k-1}TP_{s_{j}}\cdots TP_{s_{0}}Tf_{m},  \label{gem}
\end{equation}%
and clearly, $TP_{s_{j}}\cdots TP_{s_{0}}Tf_{m}$ is a homogeneous polynomial
of degree 
\begin{equation*}
s_{0}+\cdots +s_{j}+m-2k\left( j+2\right) .
\end{equation*}%
If $s_{0},\dots ,s_{n}\in \left\{ 0,\dots ,2k-1\right\} $ are given, where $%
n>m$, then $TP_{s_{n}}\cdots TP_{s_{0}}Tf_{m}$ is zero by inspection of its
degree 
\begin{equation*}
m+s_{0}+\dots +s_{n}-2k\left( n+2\right) \leq m+\left( n+1\right) \left(
2k-1\right) -2k\left( n+2\right)
\end{equation*}%
\begin{equation*}
=m-2k-n-1<0,
\end{equation*}%
so the $m$ in the first summatory of (\ref{gem}) can be replaced by $\infty $%
. By hypothesis, $P_{s}=0$ for all $s\in \left\{ \beta +1,\dots
,2k-1\right\} $, so we can write 
\begin{equation}
T_{P}\left( f_{m}\right) =\sum_{j=-1}^{\infty }\sum_{s_{0}=0}^{\beta
}\sum_{s_{1}=0}^{\beta }\dots \sum_{s_{j}=0}^{\beta }TP_{s_{j}}\cdots
TP_{s_{0}}Tf_{m}.  \label{gema}
\end{equation}%
In order to show that the sum in (\ref{eqgconvergence}) converges absolutely
and uniformly on compacta it suffices to show that 
\begin{equation*}
G:=\sum_{j=-1}^{\infty }\sum_{s_{0}=0}^{\beta }\sum_{s_{1}=0}^{\beta }\cdots
\sum_{s_{j}=0}^{\beta }\sum_{m=0}^{\infty }TP_{s_{j}}\cdots TP_{s_{0}}Tf_{m}
\end{equation*}%
does so. Then the sum can be reordered and shown to be equal to $g$. Next we
collect all summands having degree $M\geq 0$. The requirement 
\begin{equation*}
\deg TP_{s_{j}}\cdots TP_{s_{0}}Tf_{m}=s_{0}+\cdots +s_{j}+m-2k\left(
j+2\right) =M
\end{equation*}%
means that $m=M+2k\left( j+2\right) -\left( s_{0}+\cdots +s_{j}\right) $,
and therefore we consider the sum 
\begin{equation*}
G_{M}:=\sum_{j=-1}^{\infty }\sum_{s_{0}=0}^{\beta }\sum_{s_{1}=0}^{\beta
}\dots \sum_{s_{j}=0}^{\beta }TP_{s_{j}}\cdots TP_{s_{0}}Tf_{M+2k\left(
j+2\right) -\left( s_{0}+\cdots +s_{j}\right) }.
\end{equation*}%
We show next that $G_{M}$ converges absolutely everywhere. By Theorem \ref%
{thm:norm}, for every $\theta \in \mathbb{S}^{d-1}$ we have 
\begin{equation}
\left\vert TP_{s_{j}}\cdots TP_{s_{0}}Tf_{M+2k\left( j+2\right) -\left(
s_{0}+\cdots +s_{j}\right) }\left( \theta \right) \right\vert  \label{thm3}
\end{equation}%
\begin{equation}
\leq \frac{\sqrt{2}}{\sqrt{\omega _{d-1}}}\left( 1+M\right) ^{\left(
d-1\right) /2}\ \left\Vert TP_{s_{j}}\cdots TP_{s_{0}}Tf_{M+2k\left(
j+2\right) -\left( s_{0}+\cdots +s_{j}\right) }\right\Vert _{L^{2}(\mathbb{S}%
^{d-1})}.  \label{thm3b}
\end{equation}%
Recall the following notation (used in Proposition \ref{boundssphere}): 
\begin{equation*}
D_{s}=\max_{\theta \in \mathbb{S}^{d-1}}\left\vert P_{s}\left( \theta
\right) \right\vert \text{ for }s=0,\dots ,2k-1.
\end{equation*}%
Since $m\leq M+2k\left( j+2\right) $, by Proposition \ref{boundssphere} 
\begin{equation}
\left\Vert TP_{s_{j}}\cdots TP_{s_{0}}Tf_{M+2k\left( j+2\right) -\left(
s_{0}+\cdots +s_{j}\right) }\right\Vert _{L^{2}(\mathbb{S}^{d-1})}
\label{prop5}
\end{equation}%
\begin{equation}
\leq \left( M+2k\left( j+2\right) +D\right) ^{\alpha \left( j+1\right)
}C^{j+1}D_{s_{j}}\cdots D_{s_{0}}\left\Vert f_{M+2k\left( j+2\right) -\left(
s_{0}+\cdots +s_{j}\right) }\right\Vert _{L^{2}(\mathbb{S}^{d-1})}.
\label{prop5b}
\end{equation}%
By assumption $\frac{2k-\beta }{\alpha }>\rho $. Since $f$ has order $\rho $%
, the bound (\ref{eqtaylor}) entails that for every $\varepsilon >0$ there
exists a constant $A_{\varepsilon }$ such that 
\begin{equation}
\left\Vert f_{m}\right\Vert _{L^{2}(\mathbb{S}^{d-1})}\leq \frac{%
A_{\varepsilon }}{m^{\frac{m}{\rho +\varepsilon }}}  \label{eqAA}
\end{equation}%
for all natural numbers $m>0.$ We may, without loss of generality, assume
that 
\begin{equation*}
\frac{2k-\beta }{\rho +\varepsilon }-\alpha >0.
\end{equation*}
Note that the right hand side of (\ref{eqAA}) is a decreasing function of $m$
and that 
\begin{equation*}
m\geq M+2k\left( j+2\right) -\beta \left( j+1\right) =M+2k+\left( 2k-\beta
\right) \left( j+1\right) .
\end{equation*}%
We infer that 
\begin{equation}
\left\Vert f_{M+2k\left( j+2\right) -\left( s_{0}+\cdots +s_{j}\right)
}\right\Vert _{L^{2}(\mathbb{S}^{d-1})}\leq \frac{A_{\varepsilon }}{\left(
M+2k+\left( 2k-\beta \right) \left( j+1\right) \right) ^{\frac{M+2k+\left(
2k-\beta \right) \left( j+1\right) }{\rho +\varepsilon }}}.  \label{eqAAA}
\end{equation}%
Now 
\begin{eqnarray*}
&&\left( M+2k+\left( 2k-\beta \right) \left( j+1\right) \right) ^{\frac{%
M+2k+\left( 2k-\beta \right) \left( j+1\right) }{\rho +\varepsilon }} \\
&=&\left( M+2k+\left( 2k-\beta \right) \left( j+1\right) \right) ^{\frac{M+2k%
}{\rho +\varepsilon }}\left( M+2k+\left( 2k-\beta \right) \left( j+1\right)
\right) ^{\frac{\left( 2k-\beta \right) \left( j+1\right) }{\rho
+\varepsilon }} \\
&\geq &\left( M+2k\right) ^{\frac{M+2k}{\rho +\varepsilon }}\left(
M+2k+\left( 2k-\beta \right) \left( j+1\right) \right) ^{\frac{\left(
2k-\beta \right) \left( j+1\right) }{\rho +\varepsilon }}.
\end{eqnarray*}%
Furthermore, we have 
\begin{equation*}
\sum_{s_{0}=0}^{2k-1}\sum_{s_{1}=0}^{2k-1}\dots
\sum_{s_{j}=0}^{2k-1}D_{s_{j}}\cdots D_{s_{0}}=\left( D_{0}+\cdots
+D_{2k-1}\right) ^{j+1}.
\end{equation*}%
Define $\widetilde{D}:=D_{0}+\cdots +D_{2k-1}.$ Using (\ref{thm3})-(\ref%
{thm3b}), (\ref{prop5})-(\ref{prop5b}) and (\ref{eqAAA}) we get 
\begin{eqnarray}
G_{M}^{\left( j\right) } &:&=\sum_{s_{0}=0}^{2k-1}\sum_{s_{1}=0}^{2k-1}\dots
\sum_{s_{j}=0}^{2k-1}\left\vert TP_{s_{j}}\dots TP_{s_{0}}Tf_{M+2k\left(
j+2\right) -\left( s_{0}+\cdots +s_{j}\right) }\right\vert  \label{eqDefGMJ}
\\
&\leq &\frac{\sqrt{2}A_{\varepsilon }}{\sqrt{\omega _{d-1}}}\frac{\left(
M+1\right) ^{\frac{d-1}{2}}\left( M+2k\left( j+2\right) +D\right) ^{\alpha
\left( j+1\right) }C^{j+1}}{\left( M+2k\right) ^{\frac{M+2k}{\rho
+\varepsilon }}\left( M+2k+\left( 2k-\beta \right) \left( j+1\right) \right)
^{\frac{\left( 2k-\beta \right) \left( j+1\right) }{\rho +\varepsilon }}}\ 
\widetilde{D}^{j+1}.  \label{eqDefGMJ2}
\end{eqnarray}%
Note that 
\begin{equation*}
\frac{\left( M+2k\left( j+2\right) +D\right) ^{\alpha }}{\left( M+2k+\left(
2k-\beta \right) \left( j+1\right) \right) ^{\frac{\left( 2k-\beta \right) }{%
\rho +\varepsilon }}}=\frac{\left( 2k\right) ^{\alpha }}{\left( 2k-\beta
\right) ^{\frac{\left( 2k-\beta \right) }{\rho +\varepsilon }}}\frac{\left( 
\frac{M+2k+D}{2k}+j+1\right) ^{\alpha }}{\left( \frac{M+2k}{2k-\beta }%
+j+1\right) ^{\frac{\left( 2k-\beta \right) }{\rho +\varepsilon }}}.
\end{equation*}%
For all $M,j\in \mathbb{N}_{0}$ the inequality 
\begin{equation*}
\frac{M+2k+D}{2k}+j+1\leq \left( D+1\right) \left( \frac{M+2k}{2k-\beta }%
+j+1\right)
\end{equation*}%
holds since 
\begin{eqnarray*}
&&\left( D+1\right) \left( M+2k\right) 2k-\left( M+2k+D\right) \left(
2k-\beta \right) \\
&=&D\left( M+2k\right) 2k+\beta \left( M+2k+D\right) -2Dk\geq 0.
\end{eqnarray*}%
Thus, 
\begin{equation*}
\frac{\left( M+2k\left( j+2\right) +D\right) ^{\alpha }}{\left( M+2k+\left(
2k-\beta \right) \left( j+1\right) \right) ^{\frac{\left( 2k-\beta \right) }{%
\rho +\varepsilon }}}\leq \frac{\left( 2k\right) ^{\alpha }\left( D+1\right)
^{\alpha }}{\left( 2k-\beta \right) ^{\frac{\left( 2k-\beta \right) }{\rho
+\varepsilon }}}\frac{\left( \frac{M+2k}{2k-\beta }+j+1\right) ^{\alpha }}{%
\left( \frac{M+2k}{2k-\beta }+j+1\right) ^{\frac{\left( 2k-\beta \right) }{%
\rho +\varepsilon }}},
\end{equation*}%
so 
\begin{equation*}
G_{M}^{\left( j\right) }\leq \frac{\sqrt{2}A_{\varepsilon }}{\sqrt{\omega
_{d-1}}}\frac{\left( M+1\right) ^{\frac{d-1}{2}}}{\left( M+2k\right) ^{\frac{%
M+2k}{\rho +\varepsilon }}}\left( \frac{C\widetilde{D}\left( 2k\right)
^{\alpha }\left( D+1\right) ^{\alpha }}{\left( 2k-\beta \right) ^{\frac{%
\left( 2k-\beta \right) }{\rho +\varepsilon }}\left( \frac{M+2k}{2k-\beta }%
+j+1\right) ^{\frac{\left( 2k-\beta \right) }{\rho +\varepsilon }-\alpha }}%
\right) ^{j+1}.
\end{equation*}

Recall that we have choosen $\varepsilon >0$ such that $\gamma :=\frac{%
2k-\beta }{\rho +\varepsilon }-\alpha >0$. For any fixed $M\geq 0$, we can
find a natural number $j_{0}$ such that for all $j\geq j_{0}$ 
\begin{equation*}
\frac{C\widetilde{D}\left( 2k\right) ^{\alpha }\left( D+1\right) ^{\alpha }}{%
\left( 2k-\beta \right) ^{\frac{\left( 2k-\beta \right) }{\rho +\varepsilon }%
}\left( \frac{M+2k}{2k-\beta }+j+1\right) ^{\frac{\left( 2k-\beta \right) }{%
\rho +\varepsilon }-\alpha }}\leq \frac{1}{2}.
\end{equation*}%
It follows that $|G_{M}|\leq \sum_{j=0}^{\infty }G_{M}^{\left( j\right) }$
and the latter series converges, so $G_{M}$ is well-defined for every $M\in 
\mathbb{N}_{0}.$

Next we want to estimate $G_{M}$ for $M \gg 1$. Take $M_{0}$ so large such
that for all $M\geq M_{0}$%
\begin{equation*}
\frac{C\widetilde{D}\left( 2k\right) ^{\alpha }\left( D+1\right) ^{\alpha }}{%
\left( 2k-\beta \right) ^{\frac{\left( 2k-\beta \right) }{\rho +\varepsilon }%
}}\frac{1}{\frac{M+2k}{2k-\beta }^{\frac{\left( 2k-\beta \right) }{\rho
+\varepsilon }-\alpha }}\leq \frac{1}{2}.
\end{equation*}%
Then for every $M\geq M_{0}$, 
\begin{equation*}
|G_{M}| \le \sum_{j=0}^{\infty }G_{M}^{\left( j\right)} \leq \frac{\sqrt{2}%
A_{\varepsilon }}{\sqrt{\omega _{d-1}}}\frac{\left( M+1\right) ^{\frac{d-1}{2%
}}}{\left( M+2k\right) ^{\frac{M+2k}{\rho +\varepsilon }}}\sum_{j=0}^{\infty
}\frac{1}{2^{j+1}}.
\end{equation*}%
Using the criterion from (\ref{eqtaylor}) and the observation after (\ref%
{LindPring}), it follows that $G =\sum_{M=0}^{\infty }G_{M}$ is a
well-defined entire function of order $\le \rho .$ Finally, since both $f$
and $\left( P_{2k}-P_{\beta }-\cdots -P_{0}\right) g$ have order $\le \rho$,
so does $h$.
\end{proof}

If $\alpha =0$ then $\left( 2k-\beta \right) /\alpha =\infty$, so by Theorem %
\ref{Thm6}, for every entire function $f$ of order $0<\rho <\infty $ we can
find entire functions $q$ and $h$ of order $\leq \rho $ such that (\ref%
{eqdecomp11}) holds. Since we work only with upper estimates we cannot
conclude that $q$ and $h$ have exact order $\rho .$ However, if $q$ does
have order $\rho $, then the type of $q$ is bounded by the type of $f.$ When 
$\alpha >0$, Theorem \ref{Thm6} provides a Fischer decomposition only for
entire functions of order $\rho <\left( 2k-\beta \right) /\alpha $, and it
is unclear whether one can find Fischer decompositions for entire functions
of higher order. However, when $\rho =\left( 2k-\beta \right) /\alpha $, we
can prove the existence of Fischer decompositions by assuming that the type
of $f$ is sufficiently small for $\alpha > 0$ (no such restriction is needed
if $\alpha = 0$). This is the content of the next theorem. Let us also
recall the notation $D_{s}:=\max_{\theta \in \mathbb{S}^{d-1}}\left\vert
P_{s}\left( \theta \right) \right\vert .$

\begin{theorem}
\label{Thm6b}Let $P_{2k}$ be a homogeneous polynomial in $d$ variables, of
degree $2k>0$, and let $P$ be a polynomial of degree $2k$, having the form (%
\ref{eqHom}). Assume that $\left( P,\Delta ^{k}\right) $ is a Fischer pair
for $\mathcal{P}\left( \mathbb{R}^{d}\right) $. Write $T:=T_{2k}.$ Suppose
there exist $C>0$, $D>0$ and $\alpha \geq 0$ such that 
\begin{equation*}
\left\Vert Tf_{m}\right\Vert _{L^{2}(\mathbb{S}^{d-1})}\leq C\left(
m+D\right) ^{\alpha }\left\Vert f_{m}\right\Vert _{L^{2}(\mathbb{S}^{d-1})}
\end{equation*}%
for every homogeneous polynomial $f_{m}$ of degree $m$, and let $\beta <2k$
be a natural number such that $P_{j}=0$ whenever $j=\beta +1,\dots ,2k-1$.
If $\alpha =0$, then for every entire function $f$ of finite order $0<\rho
<\infty $ there exist entire functions $q$ and $h$ of order bounded by $\rho 
$, with 
\begin{equation}
f=\left( P_{2k}-P_{\beta }-\dots -P_{0}\right) q + h\text{ and }\Delta
^{k}h=0.  \label{eqdecomp12}
\end{equation}
Moreover, if $q$ has order $\rho $ then the type of $q$ is smaller than or
equal to the type of $f.$ If $\alpha >0,$ then for every entire function $f$
of order $\rho =\frac{2k-\beta }{\alpha }$ and type $\tau $ satisfying 
\begin{equation*}
\ \frac{\left( 2k\right) ^{\frac{2k-\beta }{\rho }}}{\left( 2k-\beta \right)
^{\frac{2k-\beta }{\rho }}}\ C\left( D_{0}+\cdots +D_{\beta }\right) \left(
e\rho \tau \right) ^{\frac{2k-\beta }{\rho }}<1,
\end{equation*}
there exist entire functions $q$ and $h$ of order $\leq \rho $ such that ( %
\ref{eqdecomp12}) holds. Moreover, if $q$ has order $\rho $ then the type of 
$q$ is smaller than or equal to the type of $f$.
\end{theorem}

\begin{proof}
We proceed as in the last proof, but unlike the preceding result, for which
the estimate from (\ref{eqtaylor}) sufficed, here we will need the sharper
bound given in (\ref{LindPring}): from the Theorem of Lindel\"{o}f and
Pringsheim we conclude that for every $\varepsilon >0$ there exists an $%
A_{\varepsilon }>0$ such that for every $m\geq 1$, 
\begin{equation}
\left\Vert f_{m}\right\Vert _{L^{2}(\mathbb{S}^{d-1})}\leq \frac{%
A_{\varepsilon }}{m^{\frac{m}{\rho }}}\left( e\rho \tau +\varepsilon \right)
^{m/\rho }.  \label{type}
\end{equation}%
The right hand side is clearly decreasing on $m$ for $m\geq e\rho \tau
+\varepsilon $. Arguing as in the preceding proof and using (\ref{type}),
for $j\geq e\rho \tau +\varepsilon $ or $M\geq e\rho \tau +\varepsilon $ we
have 
\begin{equation*}
\left\Vert f_{M+2k\left( j+2\right) -\left( s_{0}+\cdots +s_{j}\right)
}\right\Vert _{L^{2}(\mathbb{S}^{d-1})}\leq \frac{A_{\varepsilon }\left(
e\rho \tau +\varepsilon \right) ^{\frac{M+2k+\left( 2k-\beta \right) \left(
j+1\right) }{\rho }}}{\left( M+2k+\left( 2k-\beta \right) \left( j+1\right)
\right) ^{\frac{M+2k+\left( 2k-\beta \right) \left( j+1\right) }{\rho }}}\ .
\end{equation*}%
Following the arguments and using the notation from the last proof, we
obtain the estimate 
\begin{equation*}
G_{M}^{\left( j\right) }\leq \frac{\sqrt{2}A_{\varepsilon }}{\sqrt{\omega
_{d-1}}}\frac{\left( M+1\right) ^{\frac{d-1}{2}}\left( M+2k\left( j+2\right)
+D\right) ^{\alpha \left( j+1\right) }C^{j+1}\widetilde{D}^{j+1}}{\left(
M+2k\right) ^{\frac{M+2k}{\rho }}\left( M+2k+\left( 2k-\beta \right) \left(
j+1\right) \right) ^{\frac{\left( 2k-\beta \right) \left( j+1\right) }{\rho }%
}}\ \left( e\rho \tau +\varepsilon \right) ^{\frac{M+2k+\left( 2k-\beta
\right) \left( j+1\right) }{\rho }}.
\end{equation*}%
It follows that%
\begin{equation*}
G_{M}^{\left( j\right) }\leq \frac{\sqrt{2}A_{\varepsilon }}{\sqrt{\omega
_{d-1}}}\frac{\left( M+1\right) ^{\frac{d-1}{2}}}{\left( M+2k\right) ^{\frac{%
M+2k}{\rho }}}\left( e\rho \tau +\varepsilon \right) ^{\frac{M+2k}{\rho }%
}\left( \frac{\left( M+2k\left( j+2\right) +D\right) ^{\alpha }C\widetilde{D}%
\left( e\rho \tau +\varepsilon \right) ^{\frac{2k-\beta }{\rho }}}{\left(
M+2k+\left( 2k-\beta \right) \left( j+1\right) \right) ^{\frac{\left(
2k-\beta \right) }{\rho }}}\ \right) ^{j+1}.
\end{equation*}%
If $\alpha =0$ then there exists a $j_{0}\in \mathbb{N}$ such that for all $%
j\geq j_{0}$ and for all $M\geq 0$, 
\begin{equation*}
C\widetilde{D}\left( \frac{e\rho \tau +\varepsilon }{M+2k+\left( 2k-\beta
\right) \left( j+1\right) }\right) ^{\frac{\left( 2k-\beta \right) }{\rho }%
}\leq \frac{1}{2}.
\end{equation*}%
Thus there exists a constant $B_{j_{0}}$ such that 
\begin{equation*}
\left\vert G_{M}\right\vert \leq \sum_{j=0}^{\infty }G_{M}^{\left( j\right)
}\leq B_{j_{0}}\frac{\sqrt{2}A_{\varepsilon }}{\sqrt{\omega _{d-1}}}\frac{%
\left( M+1\right) ^{\frac{d-1}{2}}}{\left( M+2k\right) ^{\frac{M+2k}{\rho }}}%
\left( e\rho \tau +\varepsilon \right) ^{\frac{M+2k}{\rho }}.
\end{equation*}

From this we see that 
\begin{equation*}
\lim_{M\rightarrow \infty }\sup \left( M\cdot \max_{\theta \in \mathbb{S}%
^{d-1}}\left\vert G_{M}\left( \theta \right) \right\vert ^{\frac{\rho }{M}%
}\right) \leq e\rho \tau +\varepsilon
\end{equation*}%
for every $\varepsilon >0.$ Now let $\varepsilon \rightarrow 0.$ Using the
remark after (\ref{LindPring}), this implies that $G\left( x\right) $ is a
well-defined entire function of order $\leq \rho ,$ and if $G$ has order $%
\rho $ it follows that the type of $G$ is smaller or equal to $\tau .$

Now assume that $\alpha =\frac{2k-\beta }{\rho } > 0$. Then 
\begin{equation*}
\frac{\left( M+2k\left( j+2\right) +D\right) ^{\alpha \left( j+1\right) }}{%
\left( M+2k+\left( 2k-\beta \right) \left( j+1\right) \right) ^{\frac{\left(
2k-\beta \right) \left( j+1\right) }{\rho }}}=\frac{\left( 2k\right) ^{\frac{%
2k-\beta }{\rho }}}{\left( 2k-\beta \right) ^{\frac{2k-\beta }{\rho }}}\frac{%
\left( \frac{M+2k+D}{2k}+j+1\right) ^{\frac{2k-\beta }{\rho }}}{\left( \frac{%
M+2k}{2k-\beta }+j+1\right) ^{\frac{2k-\beta }{\rho }}}.
\end{equation*}%
We now choose $\delta >0$ and $\varepsilon >0$ such that%
\begin{equation*}
\ \frac{\left( 2k\right) ^{\frac{2k-\beta }{\rho }}}{\left( 2k-\beta \right)
^{\frac{2k-\beta }{\rho }}}\ C\widetilde{D}\left( 1+\delta \right) \left(
e\rho \tau +\varepsilon \right) ^{\frac{2k-\beta }{\rho }}<1.
\end{equation*}%
Given fixed $M$ and $\delta >0$ there exists $j_{M}$ such that for all $%
j\geq j_{M}$ we have 
\begin{equation*}
\frac{\frac{M+2k+D}{2k}+j+1}{\frac{M+2k}{2k-\beta }+j+1}\leq 1+\delta
\end{equation*}%
It follows that%
\begin{equation*}
G_{M}^{\left( j\right) }\leq \frac{\sqrt{2}A_{\varepsilon }}{\sqrt{\omega
_{d-1}}}\frac{\left( M+1\right) ^{\frac{d-1}{2}}}{\left( M+2k\right) ^{\frac{%
M+2k}{\rho }}}\left( e\rho \tau +\varepsilon \right) ^{\frac{M+2k}{\rho }%
}\left( \ \frac{\left( 2k\right) ^{\frac{2k-\beta }{\rho }}}{\left( 2k-\beta
\right) ^{\frac{2k-\beta }{\rho }}}\left( 1+\delta \right) C\widetilde{D}%
\left( e\rho \tau +\varepsilon \right) ^{\frac{\left( 2k-\beta \right) }{%
\rho }}\right) ^{j+1}.
\end{equation*}%
Thus the series $\sum_{j=0}^{\infty }G_{M}^{\left( j\right) }$ converges.
Next we want to estimate $G_{M}$ for large $M.$ Take $M_{0}$ so large such
that for all $M\geq M_{0}$ and for all $j\geq 0$ 
\begin{equation*}
\frac{\frac{M+2k+D}{2k}+j+1}{\frac{M+2k}{2k-\beta }+j+1}\leq 1+\delta .
\end{equation*}%
Then we have for all $M\geq M_{0}$ that 
\begin{equation*}
|G_{M}|\leq \sum_{j=0}^{\infty }G_{M}^{\left( j\right) }\leq \frac{\sqrt{2}%
A_{\varepsilon }}{\sqrt{\omega _{d-1}}}\frac{\left( M+1\right) ^{\frac{d-1}{2%
}}}{\left( M+2k\right) ^{\frac{M+2k}{\rho }}}\left( e\rho \tau +\varepsilon
\right) ^{\frac{M+2k}{\rho }}\cdot \Gamma
\end{equation*}%
where 
\begin{equation*}
\Gamma =\sum_{j=0}^{\infty }\left( \ \frac{\left( 2k\right) ^{\frac{2k-\beta 
}{\rho }}}{\left( 2k-\beta \right) ^{\frac{2k-\beta }{\rho }}}\left(
1+\delta \right) C\widetilde{D}\left( e\rho \tau +\varepsilon \right) ^{%
\frac{2k-\beta }{\rho }}\right) ^{j+1}<\infty .
\end{equation*}%
Using the remark after (\ref{LindPring}), this implies that that $G\left(
x\right) $ is a well-defined entire function of order $\leq \rho $.
\end{proof}

\section{ Estimates for the norm of $T\left( f\right) $}

The following result reduces the question about a norm estimate of $T\left(
f\right) $ to a question about an integral inequality for homogeneous
polynomials:

\begin{proposition}
\label{PropT}Let $P_{2k}$ be a homogeneous polynomial of degree $2k > 0$.
Suppose that for every $m\in \mathbb{N}_{0}$, there is a constant $C_{m}>0$
such that 
\begin{equation*}
\left\langle P_{2k}f_{m},f_{m}\right\rangle _{L^2(\mathbb{S}^{d-1})}\geq
C_{m}\left\langle f_{m},f_{m}\right\rangle _{L^2(\mathbb{S}^{d-1})}
\end{equation*}
whenever $f_{m}$ is a homogeneous polynomial of degree $m.$ Then $\left(
P_{2k},\Delta ^{k}\right) $ is a Fischer pair for $\mathcal{P}\left( \mathbb{%
R}^{d}\right) $ and 
\begin{equation*}
\left\Vert T_{P_{2k}}\left( f_{m}\right) \right\Vert _{L^2(\mathbb{S}%
^{d-1})}\leq \frac{1}{C_{m-2k}}\left\Vert f_{m}\right\Vert _{L^2(\mathbb{S}%
^{d-1})}.
\end{equation*}
\end{proposition}

\begin{proof}
First we show that $\left( P_{2k},\Delta ^{k}\right) $ is a Fischer pair for 
$\mathcal{P}\left( \mathbb{R}^{d}\right)$. By \cite[Theorem 37]{Rend08}, it
suffices to prove the injectivity of $q\longmapsto \Delta ^{k}\left(
P_{2k}q\right) .$ If $\Delta ^{k}\left( P_{2k}q_{m}\right) =0$ for some
homogeneous polynomial $q_{m}$ of degree $m$ then%
\begin{equation*}
\left\langle P_{2k}q_{m},f\right\rangle _{L^2(\mathbb{S}^{d-1})}=0
\end{equation*}%
for all polynomials $f$ with $\deg f+2k-2<m+2k,$ see Theorem 2 in \cite%
{Rend08}. Taking $f=q_{m}$, we obtain 
\begin{equation*}
0=\left\langle P_{2k}q_{m},q_{m}\right\rangle _{L^2(\mathbb{S}^{d-1})}\geq
C_{m}\left\langle q_{m},q_{m}\right\rangle _{L^2(\mathbb{S}^{d-1})}\geq 0.
\end{equation*}%
Thus $q_{m}=0$ and the Fischer operator is injective.

Next, let $f_{m}$ be a homogeneous polynomial of degree $m$. We write 
\begin{equation*}
f_{m}=P_{2k}\cdot T_{P_{2k}}\left( f_{m}\right) +h_{m},
\end{equation*}%
where $\Delta ^{k}h_{m}=0.$ If $T_{P_{2k}}\left( f_{m}\right) =0$, there is
nothing to prove, so assume otherwise. Then $m-2k\geq 0$, and $%
T_{P_{2k}}\left( f_{m}\right) $ is a homogeneous polynomial of degree $m-2k$%
, so $h_m$ is either the zero polynomial or a homogeneous polynomial of
degree $m$. Using $\Delta ^{k}h_{m}=0$ and Theorem 2 of \cite{Rend08}, we
conclude that $\left\langle h_{m}, T_{P_{2k}} \left( f_{m}\right)
\right\rangle _{L^2(\mathbb{S}^{d-1})}=0.$ Thus 
\begin{equation*}
\left\langle f_{m}, T_{P_{2k}} \left( f_{m}\right)\right\rangle _{L^2(%
\mathbb{S}^{d-1})} =\left\langle P_{2k} T_{P_{2k}} \left( f_{m}\right)
,T_{P_{2k}} \left( f_{m}\right)\right\rangle _{L^2(\mathbb{S}^{d-1})}
\end{equation*}
\begin{equation*}
\geq C_{m-2k}\left\langle T_{P_{2k}} \left( f_{m}\right) , T_{P_{2k}} \left(
f_{m}\right)\right\rangle _{L^2(\mathbb{S}^{d-1})}.
\end{equation*}
By the Cauchy-Schwarz inequality, 
\begin{equation*}
C_{m-2k}\left\Vert T_{P_{2k}} \left( f_{m}\right) \right\Vert _{L^2(\mathbb{S%
}^{d-1})}^{2}\leq \left\Vert f_{m}\right\Vert _{L^2(\mathbb{S}%
^{d-1})}\left\Vert T_{P_{2k}} \left( f_{m}\right)\right\Vert _{L^2(\mathbb{S}%
^{d-1})}.
\end{equation*}%
After dividing both sides of the inequality by $\left\Vert T_{P_{2k}} \left(
f_{m}\right) \right\Vert _{L^2(\mathbb{S}^{d-1})},$ we obtain the result.
\end{proof}

Next we note that it is enough to assume $\left\langle
P_{2k}f_{m},f_{m}\right\rangle _{L^2(\mathbb{S}^{d-1})}\geq
C_{m}\left\langle f_{m},f_{m}\right\rangle _{L^2(\mathbb{S}^{d-1})}$ for $m$
even.

\begin{lemma}
\label{Lemevenodd} Let $P_{2k}$ be a homogeneous polynomial of degree $2k >
0 $. Suppose that for each $m\in \mathbb{N}_{0}$ there exists a constant $%
C_{2m}>0 $ such that 
\begin{equation*}
\left\langle P_{2k}f_{2m},f_{2m}\right\rangle _{L^2(\mathbb{S}^{d-1})}\geq
C_{2m}\left\langle f_{2m},f_{2m}\right\rangle _{L^2(\mathbb{S}^{d-1})}
\end{equation*}%
for all homogeneous polynomials $f_{2m}$ of degree $2m$, where $m \ge 0$.
Then 
\begin{equation*}
\left\langle P_{2k}f_{2m+1},f_{2m+1}\right\rangle _{L^2(\mathbb{S}%
^{d-1})}\geq C_{2m+2}\left\langle f_{2m+1},f_{2m+1}\right\rangle _{L^2(%
\mathbb{S}^{d-1})}
\end{equation*}%
for all homogeneous polynomials $f_{2m+1}$ of degree $2m+1.$
\end{lemma}

\begin{proof}
Let $f_{2m+1}$ be a homogeneous polynomial of degree $2m+1.$ Define $%
F_{j}\left( x\right) :=x_{j}f_{2m+1}\left( x\right) $ for $x=\left(
x_{1},\dots ,x_{d}\right) \in \mathbb{R}^{d}$ and $j=1,\dots ,d.$ Recall
that $\left\vert x\right\vert ^{2}=x_{1}^{2}+\dots +x_{d}^{2}$ and $%
\left\vert \theta \right\vert ^{2}=1$ for $\theta \in \mathbb{S}^{d-1}.$
Then 
\begin{equation*}
\left\langle P_{2k}f_{2m+1},f_{2m+1}\right\rangle _{L^2(\mathbb{S}^{d-1})}
=\left\langle P_{2k}\left\vert x\right\vert
^{2}f_{2m+1},f_{2m+1}\right\rangle _{L^2(\mathbb{S}^{d-1})}
=\sum_{j=1}^{d}\left\langle P_{2k}F_{j},F_{j}\right\rangle _{L^2(\mathbb{S}%
^{d-1})}.
\end{equation*}%
Since $F_{j}$ is a homogeneous polynomial of degree $2m+2$ our assumption
implies that 
\begin{equation*}
\left\langle P_{2k}f_{2m+1},f_{2m+1}\right\rangle _{L^2(\mathbb{S}%
^{d-1})}\geq \sum_{j=1}^{d}C_{2m+2}\left\langle F_{j},F_{j}\right\rangle
_{L^2(\mathbb{S}^{d-1})}=C_{2m+2}\left\langle f_{2m+1},f_{2m+1}\right\rangle
_{L^2(\mathbb{S}^{d-1})}.
\end{equation*}
\end{proof}

Next we use the orthogonality relations for spherical harmonics in order to
derive bounds for homogeneous polynomials of even degree.

\begin{lemma}
\label{LemReduceHarmonic}Let $P_{2k}$ be a homogeneous polynomial of degree $%
2k$ and let 
\begin{equation*}
H_{2m}:=\left\{ \sum_{l=0}^{m}h_{2l}:\text{ }h_{2l}\text{ harmonic
homogeneous polynomial of degree }2l\right\}.
\end{equation*}%
Suppose that for each $m\in \mathbb{N}_{0}$ there exists a constant $%
C_{2m}>0 $ such that for all $h\in H_{2m}$, 
\begin{equation*}
\left\langle P_{2k}h,h\right\rangle _{L^2(\mathbb{S}^{d-1})}\geq
C_{2m}\left\langle h,h\right\rangle _{L^2(\mathbb{S}^{d-1})}.
\end{equation*}
Then%
\begin{equation*}
\left\langle P_{2k}f_{2m},f_{2m}\right\rangle _{L^2(\mathbb{S}^{d-1})}\geq
C_{2m}\left\langle f_{2m},f_{2m}\right\rangle _{L^2(\mathbb{S}^{d-1})}.
\end{equation*}%
for all homogeneous polynomials $f_{2m}$ of degree $2m.$
\end{lemma}

\begin{proof}
Let $f_{2m}$ be a homogeneous polynomial of degree $2m.$ By the Gauss
decomposition (see Theorem 5.5 in \cite{ABR92} or Theorem 5.7 in the 2001
edition) there exist homogeneous harmonic polynomials $h_{2l}$ of degree $2l$
for $l=0,\dots ,m$, such that 
\begin{equation}
f_{2m} (x) =\sum_{l=0}^{m}h_{2l}(x) \left\vert x\right\vert ^{2m-2l}\text{.}
\label{eqAlmansi1}
\end{equation}%
Define 
\begin{equation*}
h=\sum_{l=0}^{m}h_{2l}.
\end{equation*}%
Since the harmonic polynomials $h_{2l}$ have different degrees for different
values of $l$, the orthogonality relations for spherical harmonics yield 
\begin{equation*}
\left\langle f_{2m},f_{2m}\right\rangle _{L^2(\mathbb{S}^{d-1})}=%
\sum_{l=0}^{m}\left\langle h_{2l},h_{2l}\right\rangle _{L^2(\mathbb{S}%
^{d-1})}=\left\langle h,h\right\rangle _{L^2(\mathbb{S}^{d-1})}.
\end{equation*}%
Thus 
\begin{equation*}
\left\langle P_{2k}f_{2m},f_{2m}\right\rangle _{\mathbb{S}%
^{d-1}}=\sum_{l_{1}=0}^{m}\sum_{l_{2}=0}^{m}\left\langle
P_{2k}h_{2l_{1}},h_{2l_{2}}\right\rangle _{L^2(\mathbb{S}^{d-1})}=\left%
\langle P_{2k}h,h\right\rangle _{L^2(\mathbb{S}^{d-1})}
\end{equation*}
\begin{equation*}
\ge C_{2m}\left\langle h,h\right\rangle _{L^2(\mathbb{S}^{d-1})} =
C_{2m}\left\langle f_{2m},f_{2m}\right\rangle _{L^2(\mathbb{S}^{d-1})}.
\end{equation*}
\end{proof}

\section{The integral inequality for dimension $2$ and $P_2(x_1, x_2) =
x_2^2 $}

First we include two auxiliary results.

\begin{lemma}
\label{Lemsin}For all $n\geq 2$ one has the estimate 
\begin{equation*}
\sin \frac{\pi }{n}\geq \frac{\pi }{n+2}.
\end{equation*}
\end{lemma}

\begin{proof}
By taking $x=\pi /n$ it suffices to show that for all $0<x\leq \pi /2$ 
\begin{equation*}
\sin x\geq \frac{1}{\frac{1}{x}+\frac{2}{\pi }}=\frac{x}{1+\frac{2}{\pi }x},%
\text{ i. e. }\frac{x}{\sin x}\leq 1+\frac{2}{\pi }x.
\end{equation*}%
This inequality follows from the general inequality 
\begin{equation}
\frac{x}{\sin x}\leq 1+\frac{x}{2}\tan \frac{x}{2}\text{ for all }\left\vert
x\right\vert <\pi  \label{ineqbern}
\end{equation}%
and the fact that $\tan \frac{x}{2}\leq \tan \frac{\pi }{4}=1$ for $0<x\leq
\pi /2.$ Inequality (\ref{ineqbern}) follows from the representation (see 
\cite[p. 159]{Remm84}) 
\begin{equation*}
\frac{x}{\sin x}=1+x\sum_{n=1}^{\infty }\left( -1\right) ^{n-1}\frac{4^{n}-2%
}{\left( 2n\right) !}B_{2n}x^{2n-1}
\end{equation*}%
where $B_{2n}$ are the Bernoulli numbers. Using the trivial estimate $%
4^{n}-2\leq 4^{n}-1$ and the positivity of $\left( -1\right) ^{n-1}B_{2n}$
we obtain 
\begin{equation*}
\frac{x}{\sin x}\leq 1+x\sum_{n=1}^{\infty }\left( -1\right) ^{n-1}\frac{
4^{n}-1}{\left( 2n\right) !}B_{2n}x^{2n-1}=1+\frac{x}{2}\tan \frac{x}{2}
\end{equation*}

since $\tan x=\sum_{n=1}^{\infty }\left( -1\right) ^{n-1}\frac{4^{n}-1}{
\left( 2n\right) !}2^{2n}B_{2n}x^{2n-1},$ see \cite[p. 158]{Remm84}.
\end{proof}

The following result is well known; the proof is included for the reader's
convenience:

\begin{proposition}
\label{PropCheb}Let $P_{n}\left( \lambda \right) =\det \left( A_{n}-\lambda
I \right) $ where $A_{n}$ is the $n\times n$-matrix 
\begin{equation*}
A_{n}=\left( 
\begin{array}{ccccc}
0 & \sqrt{2} &  &  &  \\ 
\sqrt{2} & 0 & 1 &  &  \\ 
& 1 & 0 & \ddots &  \\ 
&  & \ddots & \ddots & 1 \\ 
&  &  & 1 & 0%
\end{array}
\right)
\end{equation*}
Then $P_{n}\left( \lambda \right) =2T_{n}\left( -\frac{\lambda }{2}\right) $
where $T_{n}$ is the Chebyshev polynomial of degree $n$.
\end{proposition}

\begin{proof}
Note that $P_{1}\left( \lambda \right) =-\lambda $ and $P_{2}\left( \lambda
\right) =\lambda ^{2}-2.$ For $n\geq 2$ expansion of the determinant
according to the last column yields the following recurrence relation: $%
P_{n+1}\left( \lambda \right) =-\lambda P_{n}\left( \lambda \right)
-P_{n-1}\left( \lambda \right) .$ For $n=1$ the recurrence relation is still
valid if we define $P_{0}\left( \lambda \right) =2$ since 
\begin{equation*}
P_{2}\left( \lambda \right) =\lambda ^{2}-2\text{ and }-\lambda P_{1}\left(
\lambda \right) -P_{0}\left( \lambda \right) =\lambda ^{2}-2.
\end{equation*}%
Replace $\lambda $ by $-2x.$ Then $P_{n+1}\left( -2x\right) =2xP_{n}\left(
-2x\right) -P_{n-1}\left( -2x\right) $ for $n\geq 1.$ Define $T_{n}(x):=%
\frac{1}{2}P_{n}\left( -2x\right) ,$ then $T_{n+1}\left( x\right)
=2xT_{n}\left( x\right) -T_{n-1}\left( x\right) $ for $n\geq 1$ which is the
recurrence relation of the Chebyshev polynomials. Further $P_{0}\left(
\lambda \right) =2$ implies that $T_{0}\left( x\right) =1$; clearly $%
T_{1}\left( x\right) =\frac{1}{2}P_{1}\left( -2x\right) =x.$ Thus $T_{n}$
are the Chebyshev polynomials.
\end{proof}

We recall and then prove our second main theorem:

\begin{theorem}
\label{ThmMainest2}Let $d=2$. Then for all homogeneous polynomials $f_{m}$
of degree $m$ the following inequality holds: 
\begin{equation}
\left\langle x_{2}^{2}f_{m},f_{m}\right\rangle _{L^2(\mathbb{S}^{1})}\geq 
\frac{ \pi ^{2}}{4\left( m+4\right) ^{2}}\left\langle
f_{m},f_{m}\right\rangle_{L^2(\mathbb{S}^{1})}.  \label{eqmain2}
\end{equation}
\end{theorem}

\begin{proof}
By Lemma \ref{Lemevenodd} it suffices to show that 
\begin{equation}
\left\langle x_{2}^{2}f_{2m},f_{2m}\right\rangle _{L^2(\mathbb{S}^{1})}\geq 
\frac{ \pi ^{2}}{4\left( 2m+3\right) ^{2}}\left\langle
f_{2m},f_{2m}\right\rangle_{L^2(\mathbb{S}^{1})}.  \label{eqmain3}
\end{equation}%
Clearly (\ref{eqmain3}) implies (\ref{eqmain2}) for even indices. Now (\ref%
{eqmain3}) and Lemma \ref{Lemevenodd} show that 
\begin{equation*}
\left\langle x_{2}^{2}f_{2m+1},f_{2m+1}\right\rangle _{L^2(\mathbb{S}%
^{1})}\geq \frac{\pi ^{2}}{4\left( 2m+2+3\right) ^{2}}\left\langle
f_{2m+1},f_{2m+1}\right\rangle _{L^2(\mathbb{S}^{1})}
\end{equation*}%
and (\ref{eqmain2}) holds also for odd indices. We use now polar coordinates 
$x=r\cos t$ and $y=r\sin t.$ An orthonormal basis of spherical harmonics
(restrictions of harmonic homogeneous polynomials $h_{\kappa }\left(
x,y\right) $ of degree $\kappa \geq 1$ to the unit circle) is given by 
\begin{equation*}
Y_{\kappa ,0}\left( t\right) :=\frac{1}{\sqrt{\pi }}\cos \kappa t\text{ and }%
Y_{\kappa ,1}\left( t\right) =\frac{1}{\sqrt{\pi }}\sin \kappa t
\end{equation*}%
for $\kappa \geq 1$, and for $\kappa =0$ we define $Y_{0,0}\left( t\right)
=1/\sqrt{2\pi }.$ It is convenient to set $Y_{0,1}(t) :=0.$ For every
integer $\kappa $ the following identities hold: 
\begin{eqnarray}
-4\sin ^{2}t\cdot \cos \kappa t &=&\cos \left( \kappa +2\right) t-2\cos
\kappa t+\cos \left( \kappa -2\right) t.  \label{id1} \\
-4\sin ^{2}t\cdot \sin \kappa t &=&\sin \left( \kappa +2\right) t-2\sin
\kappa t+\sin \left( \kappa -2\right) t.  \label{id2}
\end{eqnarray}%
For $k\geq 1$, replace $\kappa $ with $2k\geq 1$ in the identities (\ref{id1}%
) and (\ref{id2}) and multiply by $1/\sqrt{\pi }$. If $k\geq 2$, so $%
2k-2\geq 1$, we have 
\begin{equation}
-4\sin ^{2}t\cdot Y_{2k,s}=Y_{2k+2,s}-2Y_{2k,s}+Y_{2k-2,s}\text{.}
\label{id4}
\end{equation}%
If $k=1$ and $s=1$ then (\ref{id4}) holds using the convention that $%
Y_{0,1}=0.$ For $k=1$ and $s=0$ we obtain 
\begin{equation}
-4\sin ^{2}t\cdot Y_{2,0}=Y_{4,0}-2Y_{2,0}+\sqrt{2}Y_{0,0}\text{.}
\label{id5}
\end{equation}%
For the case $k=0$, formula (\ref{id1}) leads to 
\begin{equation}
-4\sin ^{2}t\cdot Y_{0,0}=\sqrt{2}Y_{2,0}-2Y_{0,0}.  \label{id6}
\end{equation}
By Lemma \ref{LemReduceHarmonic} it is enough to prove the inequality for $%
h\in H_{2m}$. Let us write 
\begin{equation}
h \left( \cos t,\sin t\right)
=\sum_{s=0}^{1}\sum_{k=0}^{m}c_{k,s}Y_{2k,s}\left( t\right)  \label{combin}
\end{equation}
with complex coefficients $c_{k,s}$ (using the convention that $Y_{0,1}=0).$
With (\ref{combin}) we arrive at 
\begin{equation}  \label{painfulsum}
-4\left\langle \sin ^{2}t \cdot h , h \right\rangle _{\mathbb{S}
^{1}}=\sum_{k=0}^{m}\sum_{s=0}^{1}\sum_{k_{1}=0}^{m}%
\sum_{s_{1}=0}^{1}c_{k,s} \overline{c_{k_{1},s_{1}}}\left\langle -4 \sin
^{2}t\cdot Y_{2k,s},Y_{2k_{1},s_{1}}\right\rangle _{L^2(\mathbb{S}^{1})}.
\end{equation}
Since $\left\langle \sin ^{2}t\cdot Y_{2k,s}\left( t\right)
,Y_{2k_{1},s_{1}}\left( t\right) \right\rangle _{L^2(\mathbb{S}^{1})}=0$ for 
$s\neq s_{1}$ we can express 
\begin{equation}
-4\left\langle \sin ^{2}t\cdot h , h \right\rangle _{L^2(\mathbb{S}^{1})}
=\Sigma _{0}+\Sigma _{1}
\end{equation}
where 
\begin{equation*}
\Sigma _{s}=\sum_{k=0}^{m}\sum_{k_{1}=0}^{m}c_{k,s}\overline{c_{k_{1},s}}
\left\langle -4 \sin ^{2}t\cdot Y_{2k,s},Y_{2k_{1},s}\right\rangle_{L^2(%
\mathbb{S}^{1})}
\end{equation*}
for $s=0,1.$ Let us write $c_{0}=\left( c_{0,0},\dots .,c_{m,0}\right) $. We
see from (\ref{id4}), (\ref{id5}) and (\ref{id6}) that the matrix
representing the operator ``multiplication by $- 4 \sin^2 t$'' on the space
generated by $\{Y_{0,0}, \dots, Y_{2m,0}\}$ (with respect to that basis) is $%
-2 I_{(m+1)\times (m + 1)} + A_{m+1}$, where $I_{(m+1)\times (m + 1)}$ is
the identity matrix and $A_{m + 1}$ the $(m + 1)\times ( m + 1)$-matrix 
\begin{equation*}
A_{m+1} =\left( 
\begin{array}{ccccc}
0 & \sqrt{2} &  &  &  \\ 
\sqrt{2} & 0 & 1 &  &  \\ 
& 1 & 0 & \ddots &  \\ 
&  & \ddots & \ddots & 1 \\ 
&  &  & 1 & 0%
\end{array}
\right) .
\end{equation*}
It now follows from (\ref{painfulsum}) that 
\begin{equation*}
\Sigma _{0}=-2c_{0}^{t}\overline{c_{0}}+c_{0}^{t} \ A_{m+1} \ \overline{c_{0}%
}.
\end{equation*}
Similarly, recalling that $Y_{0,1}=0,$ we see from (\ref{id4}), which is
valid for $k\geq 1$, that 
\begin{equation*}
\Sigma _{1}=-2c_{1}^{t}\overline{c_{1}}+c_{1}^{t} \ B_{m\times m } \ 
\overline{c_{1}}
\end{equation*}%
for $c_{1}=\left( c_{1,1},\dots ,c_{m,1}\right) ^{t},$ where $B_{m\times m }$
is the matrix 
\begin{equation*}
B_{m\times m }=\left( 
\begin{array}{cccc}
0 & 1 &  &  \\ 
1 & 0 & \ddots &  \\ 
& \ddots & \ddots & 1 \\ 
&  & 1 & 0%
\end{array}%
\right) .
\end{equation*}%
Let $\mu ^{\ast }$ be the maximal eigenvalue of $A_{m+1} .$ Then%
\begin{equation*}
c_{0}^{t} \ A_{m+1} \ \overline{c_{0}}\leq \mu ^{\ast }c_{0}^{t}\overline{%
c_{0}}
\end{equation*}%
for all $c_{0}\in \mathbb{R}^{m+1}.$ Since $B_{m\times m }$ is a submatrix
of $A_{m+1} $ we infer that%
\begin{equation*}
c_{1}^{t}B_{m\times m }\overline{c_{1}}\leq \mu ^{\ast }c_{1}^{t}\overline{%
c_{1}}
\end{equation*}%
for all $c_{1}\in \mathbb{R}^{m}.$ Thus 
\begin{equation}
-4\left\langle \sin ^{2}t\cdot h , h \right\rangle _{\mathbb{S} ^{1}}=\Sigma
_{0}+\Sigma _{1}\leq \left( \mu ^{\ast }-2\right) \left( c_{0}^{t}\overline{%
c_{0}} + c_{1}^{t}\overline{c_{1}}\right) .  \label{ineqdivid}
\end{equation}
By Proposition \ref{PropCheb}, 
\begin{equation*}
P_{m+1} := \left( \lambda \right) \det \left( A_{m+1} -\lambda I \right)
=2T_{m+1}\left( -\frac{\lambda }{2} \right) ,
\end{equation*}
where $T_{n}$ is the Chebyshev polynomial of degree $n$. The zeros of $%
T_{m+1}$ are given by $x_{k}:=\cos \left( \frac{2k-1}{m+1}\frac{\pi }{2}%
\right) $ for $k=1,\dots ,m+1.$ Since $T_{m+1}\left( x\right) =\frac{1}{2}%
P_{m+1}\left( -2x\right) $ clearly $-2x_{k}$ with $k=1,\dots ,m$ are the
roots of the polynomial $P_{m+1}.$ The largest zero $\mu $ of $P_{m+1}$ is
then 
\begin{equation*}
\mu ^{\ast }=-2\cos \left( \frac{2m+1}{m+1}\frac{\pi }{2}\right) .
\end{equation*}%
Dividing the inequality (\ref{ineqdivid}) by $-4$ we obtain 
\begin{equation*}
\left\langle \sin ^{2}t\cdot h , h \right\rangle _{L^1(\mathbb{S} ^{1})}\geq 
\frac{1}{2}\left( 1+\cos \frac{\left( m+\frac{1}{2}\right) \pi }{ m+1}%
\right) \left(c_{0}^{t}\overline{c_{0}} + c_{1}^{t}\overline{c_{1}} \right) .
\end{equation*}
Since $\cos \left( x\pi \right) =-\cos \left( x-1\right) \pi $ we obtain for 
$x=\frac{m+\frac{1}{2}}{m+1}$ 
\begin{equation*}
\cos \left( \frac{m+\frac{1}{2}}{m+1}\pi \right) =-\cos \left( \frac{1}{m+1} 
\frac{\pi }{2}\right) .
\end{equation*}
Now the identity $\frac{1}{2}\left( 1-\cos y\right) =\sin ^{2}\frac{y}{2}$,
together with Lemma \ref{Lemsin} and $\left\langle h , h \right\rangle _{L^2(%
\mathbb{S}^{1})} = \left(c_{0}^{t} \overline{c_{0}} + c_{1}^{t}\overline{%
c_{1}}\right)$, yield 
\begin{eqnarray*}
\left\langle \sin ^{2}t\cdot h, h \right\rangle _{L^2(\mathbb{S}^{1})} &\geq
&\left(c_{0}^{t} \overline{c_{0}} + c_{1}^{t}\overline{c_{1}}\right) \left(
\sin ^{2}\frac{ \pi }{4m+4}\right) \\
&\ge&\left\langle h , h \right\rangle _{L^2(\mathbb{S}^{1})}\cdot \left( 
\frac{\pi }{4m+4+2}\right) ^{2}.
\end{eqnarray*}
Thus we have 
\begin{equation*}
\left\langle \sin ^{2}t\cdot h , h \right\rangle _{L^1(\mathbb{S} ^{1})}\geq
\left\langle h , h \right\rangle _{L^2( \mathbb{S}^{1})}C_{2m} \text{ with }%
C_{2m}=\frac{\pi ^{2}}{4}\frac{1}{\left( 2m+3\right) ^{2}}.
\end{equation*}
\end{proof}

% Acknowledgements The author gratefully acknowledges support from the Simons Foundation.

\smallskip 
\noindent 
{\bf Data availability}   Data sharing is not applicable to this article since no data sets were generated or analyzed.

\smallskip 
\noindent 
{\bf Declarations} 

\smallskip \noindent 
{\bf Conflict of interest}  The authors declare that they have no competing interests.


\begin{thebibliography}{99}
\bibitem{AAR99}  Andrews, G. E.,  Askey, R.,  Roj, R.:  {Special Functions.}
Cambridge University Press (1999)

\bibitem{ArGa01} Armitage, D. H. ,  Gardiner, S. J.:   {Classical Potential
Theory.} Springer  London (2001) 

\bibitem{Armi04} Armitage, D.:  {The Dirichlet problem when the boundary
function is entire.} J. Math. Anal. \textbf{291},  565--577 (2004)

\bibitem{ABR92} Axler, S. , Bourdon, P.,  Ramey, W.: {Harmonic Function
Theory.} Springer New York (1992) 

\bibitem{AGV03}  Axler, S.,  Gorkin, P.,  Voss, K.:  {The Dirichlet problem on
quadratic surfaces.} Math. Comp. \textbf{73},  637--651 (2003)

\bibitem{AxRa95}  Axler, S.,  Ramey,W.:  {Harmonic Polynomials and
Dirichlet-type problems.} Proc. Amer. Math. Soc. \textbf{123}, 
3765--3773 (1995) 

\bibitem{Ba}  Baker, J. A.:  {The Dirichlet problem for ellipsoids.} Amer.
Math. Monthly \textbf{106},  829-834 (1999)

\bibitem{Beau97} Beauzamy, B. :  {Extremal products in Bombieri's norm.}
Rend. Istit. Mat. Univ. Trieste, Suppl. Vol. {\bf XXVIII }, 73--89 (1997)

\bibitem{BDS10}  Brackx, F.,  De Schepper, H.,  Sou\v{c}ek, V.:  {Fischer
Decompositions in Euclidean and Hermitean Clifford Analysis}, Archivum
Mathematicum (BRNO), Tomus {\bf 46},  301--321 (2010)

\bibitem{ChSi01}  Chamberland, M.,  Siegel, D.:  {Polynomial solutions to
Dirichlet problems. } Proc. Amer. Math. Soc. \textbf{129},  211-217  (2001)

\bibitem{deRo92}  de Boor, C.,  Ron, A.:  {The least solution for the
polynomial interpolation problem.} Math. Z. {\bf 210}, 347--378  (1992)

\bibitem{Dura03}  Durand, W.:  {On some boundary value problems on a strip
in the complex plane.} Reports on Mathematical Physics {\bf  52},  1--23 (2003)

\bibitem{Eben92}  Ebenfelt, P.:  {Singularities encountered by the
analytic continuation of solutions to Dirichlet's problem} Complex Varables
Theory Appl. \textbf{20},  75--91 (1992) 

\bibitem{EbSh95}  Ebenfelt, P.,  Shapiro, H.S.:   {The Cauchy--Kowaleskaya
theorem and Generalizations.} Commun. Partial Differential Equations, {\bf  20
},   939--960 (1995)

\bibitem{EbSh96}  Ebenfelt, P.,  Shapiro, H.S.:  {The mixed Cauchy problem
for holomorphic partial differential equations.} J. D'Analyse Math. {\bf 65},
 237--295 (1996)

\bibitem{EKS05}  Ebenfelt, P.,  Khavinson, D.,  Shapiro,  H.S.:  {Algebraic
Aspects of the Dirichlet problem.} Operator Theory: Advances and
Applications \textbf{156},  151--172 (2005)

\bibitem{EbRe08}  Ebenfelt, P.,  Render, H.:  {The mixed Cauchy problem with
data on singular conics.} J. London Math. Soc. \textbf{78},  248--266  (2008)

\bibitem{EbRe08b} Ebenfelt, P. ,  Render, H.:  {The Goursat Problem for a
Generalized Helmholtz Operator in }$R^{2}. $ Journal D'Analyse Math. \textbf{
105},  149--168 (2008),

 
\bibitem{ElNa12}  Elizar'ev, I. N.,  Napalkov, V.V.:  {Fischer Decomposition
for a Class of Inhomogeneous Polynomials.} Doklady Mathematics 
{\bf 85} (2) 196--197  (2012), see also Doklady Akademii Nauk. {\bf 443}  
(2), 156--157 (2021)

\bibitem{Fisc17}  Fischer, E.:   {\"{U}ber die Differentiationsprozesse der
Algebra.} J. f\"{u}r Mathematik {\bf 148},  1--78  (1917)

\bibitem{FNS66}  Flatto,L.,  Newman,  D.J.,  Shapiro, H.S.:  {The level curves
of harmonic functions.} Trans. Amer. Math. Soc. {\bf 123},  425--436 (1966) 

\bibitem{Frya78}  Fryant, A.:  {Growth of entire harmonic functions in }$
\mathbb{R}^{3}.$ J. Math. Anal. Appl.{\bf  66},  599--605 (1978)

\bibitem{Frya91}  Fryant, A.:  {Integral operators for harmonic functions.}
in \textquotedblright The mathematical heritage of C.F.
Gauss\textquotedblright , edited by G.M. Rassias, World Scientific,
Singapore 1991, 304--320.


\bibitem{FrSh87} Fryant, A., Shankar,  H.:  {Fourier coefficients and
growth of harmonic functions.} IJMMS {\bf 10},  433--452  (1987) 

\bibitem{Fuga80} Fugard, T.B.:  {Growth of Entire Harmonic Functions in} $%
R^{n},$ $n\geq 2. $ J. Math. Anal. Appl. {\bf 74},  286--291  (1980) 

\bibitem{Gard81}  Gardiner, S.J.:  {The Dirichlet and Neumann problems for
harmonic functions in half spaces.} J. London Math. Soc. {\bf 24}  (2),  
502--512 (1981). 

\bibitem{Gard93} Gardiner, S.J.:  {The Dirichlet problem with non-compact
boundary,} Math. Z. {\bf 213},  163--170  (1993) 

\bibitem{GaRe16} Gardiner, S.J., Render, H.:  {Harmonic functions which
vanish on a cylindrical surface,} J. Math. Anal. Appl. {\bf 433}, 
1870--1882 (2016) 

\bibitem{GaRe19} Gardiner, S.J.,   Render, H.:  {Harmonic functions
which vanish on coaxial cylinders.} J. Anal. Math. \textbf{138},
891--915 (2019)

\bibitem{Gonz}  Gonz\'{a}lez-Velasco, E.A.:  {Fourier Analysis and
Boundary value problems,} Academic Press  San Diego 1995.

\bibitem{HaSh94} Hansen, L., Shapiro, H.S.:  {Functional Equations and
Harmonic Extensions.} Complex Var. {\bf 24},  121--129  (1994) 

\bibitem{Khav97} Khavinson, D.:  {Cauchy's problem for harmonic functions
with entire data on a sphere.} Canad. Math. Bull. \textbf{40},  60--66 (1997) 

\bibitem{Khav22} Khavinson, D.:  {Harold Seymour Shapiro 1928---2021;
life in mathematics, in memoriam.} Anal. Math. Phys. {\bf 12}  no. 2, Paper
No. 63, 11 pp.  (2022)

\bibitem{KL} Khavinson,D.,   Lundberg,  E.: \emph{Linear holomorphic partial
differential equations and classical potential theory.} Amer.
Math. Soc.  Providence, RI (2018)

\bibitem{KhSh92} Khavinson, D.,  Shapiro, H.S.:  {Dirichlet's problem when
the data is an entire function.}, Bull. London Math. Soc. {\bf 24},  456--468  (1992) 

\bibitem{KLR17} Khavinson, D., Lundberg, E., Render, H.: {Dirichlet's
problem with entire data posed on an ellipsoidal cylinder.}  Potential Anal.
{\bf 46},  55--62  (2017)

\bibitem{KLR17b}  Khavinson,D., Lundberg,  E. , Render, H. : {The Dirichlet
problem for the slab with entire data and a difference equation for harmonic
functions.} Canad. Math. Bull. {\bf 60},  146--153  (2017) 

\bibitem{KhSt10}  Khavinson, D., Stylianopoulos, N.:  {Recurrence relations
for orthogonal polynomials and algebraicity of solutions of the Dirichlet
problem.} Springer \textquotedblleft Around the Research of Vladimir Maz'ya II,
Partial Differential Equations\textquotedblright  219--228 (2010)

\bibitem{Lund09} Lundberg, E.:  {Dirichlet's problem and complex
lightning bolts.} Comput. Meth. Funct. Theory \textbf{9}, 111--125  (2009)

\bibitem{LuRe11} Lundberg, E., Render, H.:  {The Khavinson-Shapiro
conjecture and polynomial decompositions.}   J. Math. Anal. Appl. {\bf 376}, 
506--513 (2011)

\bibitem{Mady21} Madych, W. R.:  {Harmonic functions in slabs and
half-spaces.} Harmonic analysis and applications,  Springer Optim.
Appl. {\bf 168},  325--359 (2021)  

\bibitem{MeSt85} Meril, A., Struppa, D.:  {Equivalence of Cauchy problems
for entire and exponential type functions.} Bull. Math. Soc. {\bf 17}, 
469--473 (1985) 

\bibitem{MeYg92} Meril, A.,  Yger,  A.:  {Probl\`{e}mes de Cauchy globaux.}
(French) [Global Cauchy problems] Bull. Soc. Math. France {\bf 120}, 
87--111  (1992) 

\bibitem{NeSh66} Newman, D.J.,  Shapiro, H.S.:  {Certain Hilbert spaces of
entire functions.}  Bull. Amer. Math. Soc. {\bf 72},  971--977  (1966) 

\bibitem{NeSh68} Newman, D.J.,  Shapiro, H.S.: {Fischer pairs of entire
functions. } Proc. Sympos. Pure. Math. II (Amer. Math. Soc., Providence, RI) 
360--369 (1968) 

\bibitem{PutSty} Putinar, M., Stylianopoulos, N.:  {Finite-term relations
for planar orthogonal polynomials.} Complex Anal. Oper. Theory {\bf 1}, 
447--456  (2007)

\bibitem{Remm84}  Remmert, R.: \emph{Funktionentheorie I,} Springer Berlin (1984)

\bibitem{Rend08}  Render, H.:  {Real Bargmann spaces, Fischer
decompositions and sets of uniqueness for polyharmonic functions.} Duke
Math. J. {\bf 142},  313--352  (2008) 

\bibitem{Rend16}  Render, H.: {A characterization of the
Khavinson-Shapiro conjecture via Fischer operators.} Potential Anal. {\bf 45},
539--543 (2016) 

\bibitem{Rend17}  Render, H.: {The Khavinson-Shapiro conjecture for
domains with a boundary consisting of algebraic hypersurfaces.} Complex
analysis and dynamical systems VII (Amer.
Math. Soc., Providence, RI) , Contemp. Math. {\bf 699}, 283--290 (2017) 


\bibitem{Saue00} Sauer,  T.: {Gr\"{o}bner Bases, H-Bases and
interpolation.} Trans. Amer. Math. Soc. {\bf 353},  2293--2308  (2000)

\bibitem{Shap89} Shapiro, H.S.:   {An algebraic theorem of E. Fischer and
the Holomorphic Goursat Problem.} Bull. London Math. Soc. {\bf 21}, 
513--537  (1989) 

\bibitem{Sici74} Siciak, J.:  {Holomorphic continuation of harmonic
functions.} Ann. Polon. Math. {\bf 29},  67--73  (1974) 

\bibitem{Widd60} Widder, D.V.:  {Functions harmonic in a strip.}  Proc.
Amer. Math. {\bf 12},  67--72  (1961) 

\bibitem{Zeil}  Zeilberger, D.:  {Chu's identity implies Bombieri's 1990
norm-inequality.} Amer. Math. Monthly {\bf 101},  894--896  (1994) 

\end{thebibliography}
\end{document}